\documentclass{article}
\usepackage[utf8]{inputenc}

\usepackage{float}
\usepackage{graphicx}
\usepackage{amsmath}
\usepackage{enumitem}
\usepackage{amscd}
\usepackage{amsthm}
\usepackage{mathtools}
\usepackage{amssymb}
\usepackage{tikz-cd}
\usepackage{caption}
\usepackage{dsfont}
\usepackage[font=small,labelfont=bf,width=1\textwidth]{caption}
\usepackage{multicol}
\usepackage{subfig}
\usepackage[left=3.5cm,top=4cm,right=3.5cm,bottom=4cm]{geometry}

\usepackage{siunitx}
\usepackage{lipsum}
\usepackage{xfrac}
\usepackage{marvosym}
\usepackage{epic}
\usepackage{wasysym}
\usepackage{epstopdf}
 
\usepackage{hyperref}
\usepackage{wrapfig}
\usepackage{enumitem}
\usepackage{enumerate}
\usepackage{color}
\usepackage[UKenglish]{datetime}

\theoremstyle{plain}
\newtheorem{theorem}{Theorem}[section]
\newtheorem{lemma}[theorem]{Lemma}
\newtheorem{corollary}[theorem]{Corollary}

\newtheorem{remark}[theorem]{Remark}

\theoremstyle{definition}
\newtheorem{definition}[theorem]{Definition}

\newcommand\scalemath[2]{\scalebox{#1}{\mbox{\ensuremath{\displaystyle #2}}}}

\mathtoolsset{showonlyrefs}

\title{Brownian motion and stochastic areas on complex partial flag manifolds with blocks of equal size}

\author{Teije Kuijper\footnote{t.kuijper@math.au.dk}}
\begin{document}

\maketitle

\begin{abstract}
    We construct a Brownian motion on complex partial flag manifolds with blocks of equal size as a matrix-valued diffusion from a Brownian motion on the unitary group.
    This construction leads to an explicit expression for the characteristic function of the joint distribution of the stochastic areas on these manifolds.
    The limit law of these stochastic areas is shown to be a multivariate Cauchy distribution with independent and identically distributed entries.
    By relating the area functionals on flag manifolds to the winding functional on the complex Stiefel manifold, we establish new results about simultaneous Brownian windings on the complex Stiefel manifold.
    To establish these results, this work introduces a new family of diffusions, which generalise both the Jacobi processes on the simplex and the Hermitian Jacobi processes.
\end{abstract}

\tableofcontents

\section*{Introduction}
The stochastic area functional of a Brownian motion was first studied in the complex plane by Paul Lévy in \cite{Levy1951}.
In recent years a general theory has been developed, see \cite{Baudoin2024-is} for a survey, which has deep connections to many different areas of geometry and probability theory, in particular random matrix theory.
Especially the setting of flag manifolds has proven to be fruitful.

Flag manifolds play important roles in many areas of mathematics, see \cite{baudoin2025fullflag} and the references therein.
Among the partial flag manifolds, the ones with blocks of equal size have the most symmetry and are particularly suitable for calculations, because one is able to work with block matrices.
It can for example be shown that they are precisely the partial flag manifolds for which a suitably defined radial process is a diffusion.

The complex partial flag manifolds with blocks of equal size can be identified with the homogeneous spaces
\begin{equation*}
    F_{m,2m,\dots ,km}(\mathbb{C}^{n}) =\frac{\mathbf{U}(n)}{\mathbf{U}(m)^{k+1}}
\end{equation*}
for $k,m\in\mathbb{N}$, where $n:=(k+1)m$ and $\mathbf{U}(n)$ is the unitary group.
Note that we recover the full flag manifolds when $m=1$.
The partial flag manifolds have a Kähler structure \cite{Borel1954}, which is essential for investigating the Brownian motion on $F_{m,2m,\dots ,km}(\mathbb{C}^n)$, its associated stochastic area functionals, its associated winding functionals and its connection to Brownian winding on the complex Stiefel manifolds.

The squared radial processes are defined as the processes
\begin{equation*}
    \Lambda_j (t):=(I_m+w^*_j(t)w_j(t))^{-1}
\end{equation*}
for $1\leq j\leq k+1$ and $t\geq 0$, where $w(t)=(w_1(t),\dots ,w_{k+1}(t))$ is a Brownian motion on $F_{m,2m,\dots ,km}(\mathbb{C}^n)$ given in certain affine coordinates.
The stochastic area processes and the squared radial processes together form a diffusion.
The squared radial processes are therefore studied in detail and also form a diffusion by themselves.
This work introduces a $(k+1)$-parameter family of diffusions, the Jacobi processes on the simplex of Hermitian matrices\footnote{It has later come to my attention that this process was already introduced in \cite{Songzi-Li2019}.}, of which the radial process squared is a special case.
These processes generalise simultaneously the Hermitian Jacobi processes \cite[\textsection 9.4.3]{doume2005} and the Jacobi processes on the simplex \cite[\textsection 2.2]{baudoin2025fullflag}.

The Jacobi processes on the simplex of Hermitian matrices are related to discrete \emph{random Positive Operator-Valued Measures (random POVMs)}.
In quantum theory, POVMs can be used to describe observables, see for example \cite[Chapter 3]{heinosaari2011mathematical}.
Random POVMs play a key role in quantum information theory \cite{Heinosaari2020}.
This viewpoint will not be explored in this paper.

This paper is part of a larger project, namely to determine the properties of horizontal Brownian motion, the area process and the winding process on the compact homogeneous symplectic manifolds, in particular the heat kernel, characteristic function and the asymptotics.
All compact homogeneous symplectic manifolds are Kähler, and the simply connected ones turn out to be precisely products of the complex flag manifolds \cite{Borel1954}.
Similar studies have already been carried out for the complex projective spaces \cite{Baudoin2011TheSH, baudoin2024lawindexbrownianloops}, the complex Grassmannians \cite{Baudoin2021} and the complex full flag manifolds \cite{baudoin2025fullflag}.
The setup of this paper is very similar to \cite{baudoin2025fullflag} and generalises the results of that paper on the full flag manifolds to a more general class of partial flag manifolds, namely those with blocks of equal size.

The main novel results of this extension will now be discussed.
In \textsection\ref{subsec:generator-BM-flag}, a Brownian motion on $F_{m,2m,\dots ,km}(\mathbb{C}^n)$ is represented in terms of a unitary Brownian motion.
This leads to the expression of the Laplace--Beltrami operator on $F_{m,2m,\dots ,km}(\mathbb{C}^n)$ in affine coordinates of theorem \ref{thm:Laplace-Beltrami-operator-flags}.
In theorem \ref{thm:horizontal-MB-m} this Brownian motion is lifted to a horizontal Brownian motion on the unitary group in some trivialisation.
In contrast to the full flag manifold, the first coordinate of the Brownian motion in the trivialisation is not only determined by the area processes.

Theorem \ref{thm:Laplace-transform-area-functional-equal-size} expresses the characteristic function of the area processes conditional on the radial processes squared in terms of the heat kernel of the Jacobi processes on the simplex of Hermitian matrices.
This expression is then used in corollary \ref{cor:limit-area-process-on-Fm} to show that the large time limit of the area processes are independent and identically distributed Cauchy random variables of parameter $m(n-m)$.

Definition \ref{def:matrix-Jacobi-process-simplex} introduces the Jacobi operators on the simplex of Hermitian matrices.
A thorough study of these processes is planned for a follow up article.
In theorem \ref{thm:generator-process-of-interest} it is shown that the radial processes of a Brownian motion on $F_{m,2m,\dots ,km}(\mathbb{C}^n)$ form a Jacobi process on the simplex of Hermitian matrices of index $(m/2,\dots ,m/2)$.

In \textsection\ref{sec:canonical-torus-bundle} an explicit construction of the canonical torus bundle over $F_{m,2m,\dots ,km}(\mathbb{C}^n)$ is given in terms of the quotient of Lie groups
\begin{equation*}
    P_{m,2m,\dots ,km}(\mathbb{C}^n) =\frac{\mathbf{U}(n)}{\mathbf{SU}(m)^{k+1}} .
\end{equation*}
Note that this is only non-trivial for $m>1$.
Using a trivialisation of this torus bundle, the winding forms and processes are defined in definition \ref{def:winding-form-blocks-of-equal-size} as the differential of the first coordinates of this trivialisation.
In corollary \ref{cor:connection-process-BM}, the winding processes are shown to be related to the area processes and their large time asymptotics are seen to be the same.
Furthermore, the winding processes are shown to be equal in distribution to the complex argument of the determinants of the blocks in the last row of a unitary Brownian motion.

By relating these winding processes to simultaneous winding on the Stiefel manifolds $V_{n,m}(\mathbb{C})$, with $n$ as above, the same large time asymptotics are shown to hold for the simultaneous winding as well in theorem \ref{thm:simulteneous-winding-Stiefel-manifold}.
This result generalises both \cite[\textsection 5]{baudoin2025fullflag} and \cite[\textsection 8.3.3]{Baudoin2024-is}.
Furthermore, a new interpretation of the latter result is given.

The paper is organised as follows: section \ref{chap:preliminaries} briefly reviews complex partial flag manifolds and introduces the canonical torus bundle over the partial flag manifold with blocks of equal size; section \ref{chap:Brownian-motion-F} constructs a Brownian motion on the partial flag manifolds with blocks of equal size as a projection of a unitary Brownian motion and analyses its radial dynamics; section \ref{chap:stochastic-area} defines and constructs horizontal Brownian motion on the unitary group with respect to the fibration defining the partial flag manifolds with blocks of equal size, furthermore it introduces the stochastic area processes, derives their simultaneous characteristic function and determines its asymptotics; and lastly section \ref{chap:simulteneous-winding-Stiefel} applies the previous results to Brownian winding on certain Stiefel manifolds.

\section{Preliminaries}\label{chap:preliminaries}
In this section the construction of the geometric object of interest, specific partial flag manifolds are recalled and the canonical torus bundle is constructed over them.
The canonical torus bundle places our investigation into the general context discussed in \cite[\textsection 3.5]{Baudoin2024-is}.

\subsection{Complex flag manifolds}\label{sec:complex-flag-manifold}
In this section we review the geometric structure of certain complex flag manifolds.
This subsection closely follows \cite[\textsection 2.1]{baudoin2025fullflag}, where the same constructions are carried out for the special case of the full flag manifold.

\begin{definition}
    A \emph{flag} of $\mathbb{C}^n$ is a sequence
    \begin{equation*}
        \{ 0\} =:W_0\subsetneq W_1\subsetneq\dots\subsetneq W_k\subsetneq W_{k+1} :=\mathbb{C}^n
    \end{equation*}
    of complex subspaces of $\mathbb{C}^n$.
    The number $k$ is called the \emph{length} of the flag and the $k$-tuple
    \begin{equation*}
        (\dim (W_1),\dots ,\dim (W_k))
    \end{equation*}
    the \emph{signature} of the flag.
    Furthermore the numbers $(n_1,\dots ,n_k,n_{k+1})$, with $n_{j}:=\dim (W_j) -\dim (W_{j-1})$, are called the \emph{block sizes} of the flag.
\end{definition}

\begin{definition}
    The \emph{(complex) partial flag manifold $F_{m,2m\dots ,km}(\mathbb{C}^{(k+1)m})$ of length $k$ with blocks of size $m$} is the collection of flags of $\mathbb{C}^{(k+1)m}$ of signature $(m,2m\dots ,km)$.
\end{definition}

From now on we will fix $m,k\in\mathbb{Z}_{\geq 1}$ and define $n:=(k+1)m$.
Similarly to \cite[\textsection 2.1]{baudoin2025fullflag} we can make the identification
\begin{equation}\label{eq:fundamental-identification}
    F_{m,2m,\dots ,km}(\mathbb{C}^{n})\cong\mathrm{GL}_{n}(\mathbb{C})/ T^m_n(\mathbb{C}),
\end{equation}
where $T^m_n(\mathbb{C})$ is the group of matrices $M\in\mathrm{GL}_n(\mathbb{C})$ such that $M_{ij} =0$ when $i>\ell m$ whenever $(\ell -1)m < j\leq \ell m$.
If we write the matrix in terms of blocks of size $m\times m$, the matrices in $T_n^m (\mathbb{C})$ will be upper triangular.
In particular, the complex dimension is given by
\begin{equation*}
    n^2 -m^2\sum_{j=1}^{k+1}j =n^2 -m^2\frac{(k+2)(k+1)}{2} =\frac{m^2 k(k+1)}{2}
    =\frac{n(n-m)}{2} .
\end{equation*}

The identification \eqref{eq:fundamental-identification} defines a smooth structure on $F_{m,2m,\dots ,km}(\mathbb{C}^n)$, which is the unique one making the canonical projection into a smooth submersion.
This identification can be interpreted by seeing the first $jm$ columns of an invertible matrix as a basis for $W_j$.
Two invertible matrices then give rise to the same flag if they differ by right multiplication of an upper triangular block matrix in $T_n^m(\mathbb{C})$.

Similarly, if we restrict ourselves to orthonormal bases we obtain the identification, which is the compact realisation of \eqref{eq:fundamental-identification},
\begin{equation}\label{eq:fibration-Um}
    F_{m,2m,\dots ,km}(\mathbb{C}^{n})\cong\mathbf{U}(n)/\mathbf{U}(m)^{k+1},
\end{equation}
where
\begin{equation*}
    \mathbf{U}(n):=\{ M\in\mathbb{C}^{n\times n}\mid M^*M=I_n\}
\end{equation*}
is the \emph{unitary group} and $\mathbf{U}(m)^{k+1}$ is identified with the set of diagonal block matrices in $\mathbf{U}(n)$.
The Riemannian metric on $F_{m,2m,\dots ,km}(\mathbb{C}^n)$ is the unique one making the canonical projection $\pi :\mathbf{U}(n)\rightarrow\mathbf{U}(n)/\mathbf{U}(m)^{k+1}$ into a Riemannian submersion, where $\mathbf{U}(n)$ is equipped with its bi-invariant metric induced by the Killing form.

Recall that the Grassmannian $Gr_{n,m}(\mathbb{C})$ is the space of all linear subspaces in $\mathbb{C}^n$ of dimension $m$.
Analogues to the full flag manifold, one can see $F_{m,2m,\dots ,km}(\mathbb{C}^n)$ as an algebraic subvariety of $Gr_{n,m}(\mathbb{C})\times\dots\times Gr_{n,m}(\mathbb{C})$ with the help of the Riemannian immersion
\begin{equation*}
    (W_1,\dots ,W_k)\mapsto (W_1, W_2\cap W_1^{\perp} ,\dots ,W_{k}\cap W_{k-1}^{\perp} ,W_{k}^{\perp}) .
\end{equation*}

We will parametrise (a dense subset of) $F_{m,2m,\dots ,km}(\mathbb{C}^n)$ using \emph{local affine coordinates} on the domain
\begin{equation*}
    \mathcal{D}_m :=\scalemath{0.9}{\left\{\begin{pmatrix} A_{11} &\dots &A_{1(k+1)} \\ \vdots & \ddots & \vdots \\ A_{(k+1)1} & \dots & A_{(k+1)(k+1)}\end{pmatrix}\in\mathbf{U}(n)\Biggm| A_{\ell j}\in\mathbb{C}^{m\times m} ,\det (A_{(k+1)j})\neq 0, 1\leq \ell ,j\leq k+1\right\}}
\end{equation*}
by the smooth map $p:\mathcal{D}_{m}\rightarrow\mathbb{C}^{(n-m)\times m}\times\dots\times\mathbb{C}^{(n-m)\times m}$ defined by
\begin{equation}\label{eq:Riemannian-submersion-coordinates}
    p\left(\begin{pmatrix} A_{11} &\dots &A_{1(k+1)} \\ \vdots & \ddots & \vdots \\ A_{(k+1)1} & \dots & A_{(k+1)(k+1)}\end{pmatrix}\right) :=\left(\begin{pmatrix} A_{11} \\ \vdots \\ A_{k1}\end{pmatrix}A_{(k+1)1}^{-1} ,\dots ,\begin{pmatrix} A_{1(k+1)} \\ \vdots \\ A_{k(k+1)}\end{pmatrix}A_{(k+1)(k+1)}^{-1}\right) .
\end{equation}
It is not difficult to see that for every $M, N\in\mathcal{D}_{m}$, $p(M)=p(N)$ is equivalent to $N=Mg$ for some $g\in \mathbf{U}(m)^{k+1}$ (any two such matrices differ by a diagonal unitary block matrix).
Since $p$ is a submersion from $\mathcal{D}_m$ onto its image $\mathcal{O}_m :=p(\mathcal{D}_m)$, one deduces that there exists a diffeomorphism $\Psi$ between an open dense subset of $\mathbf{U}(n)/\mathbf{U}(m)^{k+1}$ and $\mathcal{O}_m$ such that $\Psi\circ\pi =p$.
This gives rise to a local set of coordinates on $F_{m,2m,\dots ,km}(\mathbb{C}^n)$.
These coordinates are compatible with the metric in the sense that $p$ is a Riemannian submersion, which implies that $\Psi$ is an isometry.
These coordinates are only implicit and we will mostly work with the related parametrisation
\begin{align*}
    w_{j}:=(u_{1 j}u_{(k+1)j}^{-1},\cdots ,u_{kj}u_{(k+1)j}^{-1})\textrm{ for } 1\leq j\leq k+1
\end{align*}
for $U=(u_{\ell j})_{1\leq\ell ,j\leq n}\in\mathbf{U}(n)$ instead.

Note that $\mathcal{O}_m$ can explicitly be described as
\begin{equation*}
    \mathcal{O}_m =\{ (w_1,\dots ,w_{k+1})\in\mathbb{C}^{km\times m}\times\dots\times\mathbb{C}^{km\times m}\mid w_{j}^*w_{\ell} =-I_m ,1\leq j<\ell\leq k+1\} ,
\end{equation*}
which yields a parametrisation of a dense open set of $F_{m,2m,\dots ,km}(\mathbb{C}^{n})$ by the algebraic manifold $\mathcal{O}_m$.
This parametrisation and the corresponding block notation for $\mathbf{U}(n)$ will be used in the subsequent sections, unless noted otherwise. 

\subsection{The canonical torus bundle}\label{sec:canonical-torus-bundle}
For $m>1$, the fibration \eqref{eq:fibration-Um} is not a torus bundle over the flag manifold $F_{m,2m,\dots ,km}(\mathbb{C}^n)$.
To define the winding forms over $F_{m,2m,\dots ,km}(\mathbb{C}^n)$, one can construct a torus bundle over $F_{m,2m,\dots ,km}(\mathbb{C}^n)$.
The construction can straightforwardly be generalised to general partial flag manifolds.

The \emph{canonical torus bundle} over $F_{m,2m,\dots ,km}(\mathbb{C}^n)$ is defined as the quotient space
\begin{equation*}
    P_{m,2m,\dots ,km}(\mathbb{C}^n) :=\mathbf{U} (n)/\mathbf{SU}(m)^{k+1} ,
\end{equation*}
where
\begin{equation*}
    \mathbf{SU}(m):=\{ U\in\mathbf{U}(m)\mid\det (U)=1\}
\end{equation*}
is the \emph{special unitary group}.

The inclusion of Lie groups $\mathbf{SU}(m)^{k+1}\subseteq \mathbf{U}(m)^{k+1}\subseteq\mathbf{U}(n)$ gives rise to the following commutative diagram of totally geodesic fibrations

\begin{equation}\label{eq:cd-canonical-torus-bundle}
\begin{tikzcd}[ampersand replacement=\&]
	\& {\mathbf{SU}(m)^{k+1}} \\
	{\mathbf{U}(m)^{k+1}} \& {\mathbf{U}(n)} \\
	{\mathbf{U}(1)^{k+1}} \& {P_{m,2m,\dots ,km}(\mathbb{C}^n)} \& {F_{m,2m,\dots ,km}(\mathbb{C}^n)}
	\arrow[from=1-2, to=2-1]
	\arrow[from=1-2, to=2-2]
	\arrow[from=2-1, to=2-2]
	\arrow[from=2-1, to=3-1]
	\arrow["{\pi_P}", from=2-2, to=3-2]
	\arrow["{\pi_{F}}", from=2-2, to=3-3]
	\arrow[from=3-1, to=3-2]
	\arrow["\pi", from=3-2, to=3-3]
\end{tikzcd} .\end{equation}

Explicitly, the action of $\mathbf{U}(1)^{k+1}$ on $P_{m,2m,\dots ,km}(\mathbb{C}^n)$ is obtained by projecting the action on $\mathbf{U}(n)$ given by right multiplication down to $P_{m,2m,\dots ,km}(\mathbb{C}^n)$, where we see $\mathbf{U}(1)^{k+1}$ as the set of all diagonal $(k+1)\times (k+1)$ unitary matrices tensored with $I_m$.

The vertical space $\mathcal{V}_{F} :=\ker (d\pi_F)$ of the fibration $\mathbf{U}(m)^{k+1}\rightarrow\mathbf{U}(n)\rightarrow F_{m,2m,\dots ,km}(\mathbb{C}^n)$ is given by $(\mathcal{V}_{F})_M=M\mathfrak{u}(m)^{k+1}$.
Its orthogonal complement, the horizontal space $\mathcal{H}_F$ of the fibration, by
\begin{align*}
    (\mathcal{H}_{F})_M=&\left\{ N\in T_{M}\mathbf{U}(n)\mid \mathrm{Tr}(N^*MA) =0\textrm{ for all } A\in\mathfrak{u}(m)^{k+1}\right\} \\
    =&\left\{ N\in T_{M}\mathbf{U}(n)\biggm| \sum_{\ell =1}^{k+1}(N_{\ell j})^*M_{\ell j}=0\textrm{ for }1\leq j\leq k+1\right\} .
\end{align*}
We can identify the tangent space $TF_{m,2m,\dots ,km}(\mathbb{C}^n)$ with $\mathcal{H}_F$ with the help of the map $d\pi_F$.
Similarly, $T_MP_{m,2m,\dots ,km}(\mathbb{C}^n)$ can be identified with the orthogonal complement of $(\mathcal{V}_P)_M =M\mathfrak{su}(m)^{k+1}$ using $d\pi_P$.
Using these identifications the vertical space $\mathcal{V}$ of the fibration
\begin{equation}\label{eq:fibration-canonical-torus-bundle}
    \mathbf{U}(1)^{k+1}\rightarrow P_{m,2m,\dots ,km}(\mathbb{C}^{n})\rightarrow F_{m,2m,\dots ,km}(\mathbb{C}^n) ,
\end{equation}
is given by $\mathcal{V}_M =iM_1\mathbb{R}\oplus\dots\oplus iM_{k+1}\mathbb{R}$, where $M_j$ denotes the $j^{\mathrm{th}}$-column of $M$ in block notation.
To see this note that the fibres are $(k+1)$-dimensional and that $iM_1\oplus 0\oplus\dots\oplus 0,\dots , 0\oplus\dots 0\oplus iM_{k+1}$ lie in both $T_MP_{m,2m,\dots ,km}(\mathbb{C}^n)$ and the kernel of $d\pi$.

We will consider the following vector fields on $P_{m,2m,\dots ,km}(\mathbb{C})$ given at a point $[M]$ by
\begin{equation*}
    \frac{\partial f}{\partial\theta_j} ([M]):=\frac{d}{dt}\bigg|_{t=0} f\left(\left[ Me^{it\frac{E_{jj}}{m}}\right]\right) ,
\end{equation*}
where $E_{jj}:=(\delta_{(j,j)}(p,q) I_m)_{1\leq p,q\leq k+1}$.
This is well defined because the commutator of two elements in $\mathbf{U}(n)$ lies in $\mathbf{SU}(n)$.
Note that the vector fields $\frac{\partial}{\partial\theta_j}$ for $1\leq j\leq k+1$ commute and form at any point a basis of the vertical space $\mathcal{V}$.

A local trivialisation of this torus bundle is given by projecting the map
\begin{equation}\label{eq:trivialisation-canonical-torus-bundle}
\begin{split}
    \mathbb{R}^{k+1}\times\mathcal{O}_m & \rightarrow\qquad\qquad\qquad\qquad\qquad\qquad\qquad\qquad\mathcal{D}_m\\
    (\theta\quad ,\quad w)\, &\mapsto
    \begin{pmatrix}
        e^{i\frac{\theta_1}{m} }w_{11}(I_m+w_{1}^*w_{1})^{-\frac{1}{2}} & \dots & e^{i\frac{\theta_{k+1}}{m} }w_{1(k+1)}(I_m+w_{k+1}^*w_{k+1})^{-\frac{1}{2}} \\
        \vdots & \ddots & \vdots \\
        e^{i\frac{\theta_1}{m} }w_{k1}(I_m+w_{1}^*w_{1})^{-\frac{1}{2}} &\dots &e^{i\frac{\theta_{k+1}}{m} }w_{k(k+1)}(I_m+w_{k+1}^*w_{k+1})^{-\frac{1}{2}} \\
        e^{i\frac{\theta_1}{m} }(I_m+w_{1}^*w_{1})^{-\frac{1}{2}} &\dots & e^{i\frac{\theta_{k+1}}{m} }(I_m+w_{k+1}^*w_{k+1})^{-\frac{1}{2}}
    \end{pmatrix} ,
    \end{split}
\end{equation}
down to a dense open subset of $P_{m,2m,\dots ,km}(\mathbb{C}^n)$.

\begin{definition}\label{def:winding-form-blocks-of-equal-size}
    The one-forms $(d\theta_1 ,\dots ,d\theta_{k+1})$ on $\pi_P (\mathcal{D}_m)$ are called the \emph{winding forms} around $P_{m,2m,\dots ,km}(\mathbb{C}^n)\backslash\pi_P (\mathcal{D}_m)$.
\end{definition}

\begin{lemma}\label{lemma:connection-form-canonical-torus-bundle}
    The $\mathbb{R}^{k+1}$-valued one form $\eta =(\mathrm{Tr} (\omega_1),\dots ,\mathrm{Tr} (\omega_{k+1}))$ on $\mathbf{U}(n)$, with
    \begin{equation*}
        \omega_j :=\frac{1}{2i}\sum_{\ell =1}^{k+1}( M_{\ell j}^*dM_{\ell j} -dM^*_{\ell j}M_{\ell j} )\textrm{ for } 1\leq j\leq k+1 ,
    \end{equation*}
    is the connection form of the canonical torus bundle
    \begin{equation*}
        \mathbf{U}(1)^{k+1}\rightarrow P_{m,2m,\dots ,km} (\mathbb{C}^n)\rightarrow F_{m,2m,\dots ,km}(\mathbb{C}^n) .
    \end{equation*}
\end{lemma}
\begin{proof}
    The $\mathfrak{u}(m)$-valued one-form $\omega_j$ is the connection form of the Stiefel fibration $\mathbf{U}(m)\rightarrow V_{n,m}(\mathbb{C})\rightarrow G_{n,m}(\mathbb{C})$, see \cite[Lemma 8.2.1.]{Baudoin2024-is}, where
    \begin{equation*}
        V_{n,m}(\mathbb{C}):=\{ M\in\mathbb{C}^{n\times m}\mid M^*M=I_{m}\}
    \end{equation*}
    is the Stiefel manifold.
    Therefore its kernel is the horizontal space of the Stiefel submersion $V_{n,m}(\mathbb{C})\rightarrow G_{n,m}(\mathbb{C})$, it is $\mathbf{U}(m)$-invariant and satisfies $\omega_j (\hat{X}) =X$ for all $X\in\mathfrak{u}(m)$, where
    \begin{equation*}
        \hat{X}f(M):=\lim_{t\rightarrow 0}\frac{f(Me^{tX}) -f(M)}{t} .
    \end{equation*}
    The latter shows that
    \begin{equation*}
        \mathrm{Tr}\left(\omega_j\left(\frac{\partial}{\partial\theta_i}\right)\right) =\mathrm{Tr}\left(\frac{E_{jj}}{m}\right) =1.
    \end{equation*}
    It is then easily shown that
    \begin{itemize}
        \item[i)] $g^*\eta =\eta$ for all $g\in\mathbf{U}(1)^{k+1}$;
        \item[ii)] $\mathrm{Tr}\left(\omega_j\left(\frac{\partial}{\partial\theta_i}\right)\right) =\delta_{ij}$;
        \item[iii)] $\ker(\eta ) =\mathcal{H}$.
    \end{itemize}
    From which it follows that $\eta$ is the connection form.
\end{proof}

In the parametrisation \eqref{eq:trivialisation-canonical-torus-bundle}, the connection form admits the following decomposition
\begin{equation}\label{eq:fundamental-decomposition-m}
    \eta_j =d\theta_j -\frac{1}{2i}\mathrm{Tr}\left(\left( I_m+\sum_{\ell =1}^{k+1}w_{\ell j}^*w_{\ell j}\right)^{-1}\sum_{\ell =1}^{k+1} (dw^*_{\ell j}w_{\ell j} -w^*_{\ell j}dw_{\ell j})\right)
\end{equation}
as can be shown by a tedious, but straightforward, calculation.

\begin{remark}\label{rmk:torus-bundle-Boothby-Wang}
    The canonical torus bundle is related to the \emph{Boothby--Wang fibration} \cite{Boothby-Wang1958}.
    Explicitly, the Boothby--Wang fibration can be realised as
    \begin{equation*}
        \mathbf{U}(1)\rightarrow \mathbf{U}(n)/\mathbf{S}(\mathbf{U}(m)^{k+1})\rightarrow F_{m,2m,\dots,km}(\mathbb{C}^{n}) ,
    \end{equation*}
    where
    \begin{equation*}
        \mathbf{S}(\mathbf{U}(m)^{k+1}) :=\{ U\in\mathbf{U}(m)^{k+1}\mid\det (U)=1\} .
    \end{equation*}
    The connection form of the Boothby--Wang fibration can then be identified with the sum of the components of the connection form $\eta$ of lemma \ref{lemma:connection-form-canonical-torus-bundle}. 
\end{remark}

\section{Brownian motion on partial flag manifolds}\label{chap:Brownian-motion-F}
In this section we describe Brownian motion on the partial flag manifold $F_{m,2m,\dots ,km} (\mathbb{C}^n)$ in terms of a unitary Brownian motion.
We begin by briefly recalling the unitary Brownian motion and rewriting it in block form.
In the second half of this section we determine the generator of the radial process of the Brownian motion and show that it is a diffusion.
Moreover, the generator of this diffusion is determined and is a special case of what we will call a Jacobi process on the simplex of Hermitian matrices.

\subsection{Unitary Brownian motion}
Recall that the Lie algebra of the unitary group is given by
\begin{equation*}
    \mathfrak{u}(n)=\{ A\in\mathbb{C}^{n\times n}\mid A^*=-A\} .
\end{equation*}
Let $(A(t))_{t\geq 0}$ be Brownian motion on $\mathfrak{u}(n)$, for its explicit construction in terms of an Euclidean Brownian motion, see \cite[\textsection 3.1]{baudoin2025fullflag}.
Brownian motion on the unitary group $\mathbf{U}(n)$ is the solution $(U(t))_{t\geq 0}$ to the Stratonovich differential equation
\begin{equation}\label{eq:Brownian_motion_on_Un}
    dU(t) =U(t)\circ dA(t) ,
\end{equation}
where $(A(t))_{t\geq 0}$ is Brownian motion on the Lie algebra $\mathfrak{u}(n)$.
Recall the block notations $U(t)=(U_{ij}(t))_{1\leq i,j\leq k+1}$ and $A(t)=(A_{ij}(t))_{1\leq i,j\leq k+1}$ where $U_{ij},A_{ij}\in\mathbb{C}^{m\times m}$ for $1\leq i,j\leq k+1$.
Note that
\begin{equation*}
    (dA_{ij}dA_{jr})_{pq} 
    =\sum_{\ell =1}^{m}(dA_{ij})_{p\ell}(dA_{jr})_{\ell q}
    =-\sum_{\ell =1}^{m}(dA_{ij})_{p\ell}(d\overline{A}_{rj})_{q\ell}
    =-2m \delta_{pq}\delta_{ir}dt
\end{equation*}
and therefore
\begin{equation*}
   dA_{ij}dA_{jr} = -2m\delta_{ir} I_{m} dt.
\end{equation*}

\subsection{Generator of the Brownian motion on the flag manifold}\label{subsec:generator-BM-flag}
In this section we compute the generator of the Brownian motion on the partial flag manifold with blocks of equal size $F_{m,2m,\dots ,km}(\mathbb{C}^n)\cong \mathbf{U}(n)/\mathbf{U}(m)^{k+1}$.
We use the parametrisation explained in \textsection\ref{sec:complex-flag-manifold}.
Let $(U(t))_{t\geq 0}$ be a Brownian motion on $\mathbf{U}(n)$ started from a point $U(0)\in\mathcal{D}_m$.
Since the map $p:\mathcal{D}_m\rightarrow\mathcal{O}_m$ defined in \eqref{eq:Riemannian-submersion-coordinates} is a Riemannian submersion with totally geodesic fibres and since
\begin{equation*}
    \mathbb{P}(\exists t\geq 0[U(t)\notin\mathcal{D}_m]) =0,
\end{equation*}
we see that the process $(w(t))_{t\geq 0}$, defined by
\begin{equation*}
    w_{\ell j}(t):=U_{\ell j}(t)U_{(k+1)j}^{-1}(t)\textrm{ for } 1\leq\ell\leq k\textrm{ and } 1\leq j\leq k+1
\end{equation*}
in block form, parametrises a Brownian motion on $F_{m,2m,\dots ,km}(\mathbb{C}^n)$.
We define the $j^{\textrm{th}}$-column of $w(t)$ by
\begin{equation*}
    w_j(t):=(w_{1j}(t),\dots ,w_{kj}(t))^T\textrm{ for } 1\leq j\leq k+1 .
\end{equation*}
and the components of the process $(J(t))_{t\geq 0}$ by
\begin{equation}\label{eq:radial-process-squared}
    J_j(t):=w_j^*(t)w_j(t)\textrm{ for } 1\leq j\leq k+1.
\end{equation}

\begin{theorem}\label{thm:Laplace-Beltrami-operator-flags}
    The Laplace--Beltrami operator on $F_{m,2m,\dots ,km}(\mathbb{C}^n)$ is given by
    \begin{align*}
        \Delta_{F_{m,2m,\dots ,km}(\mathbb{C}^n)} 
        =&4\sum_{j=1}^{k+1}\sum_{p,r =1}^{n-m}\sum_{s ,q=1}^m\left( I_{n-m}+w_jw_j^*\right)_{pr}\left( I_{m}+w_j^*w_j\right)_{s q}\frac{\partial^2}{\partial w_{j,pq}\partial\overline{w}_{j,rs }} \\
        &-2\sum_{1\leq j\neq\ell\leq k+1}\sum_{p,r =1}^{n-m}\sum_{s,q =1}^{m}\left( w_{\ell}-w_j\right)_{ps }
        \left( w_j-w_{\ell}\right)_{rq}\frac{\partial^2}{\partial w_{j,pq}\partial w_{\ell ,rs}}\\
        &-2\sum_{1\leq j\neq\ell\leq k+1}\sum_{p,r=1}^{n-m}\sum_{s,q =1}^{m}\left( \overline{w}_{\ell}-\overline{w}_j\right)_{ps}
        \left( \overline{w}_j -\overline{w}_{\ell}\right)_{rq}\frac{\partial^2}{\partial\overline{w}_{j,pq}\partial\overline{w}_{\ell ,rs}} .
    \end{align*}
\end{theorem}
\begin{proof}
    Note that the generator of $(w(t))_{t\geq 0}$ is $\frac{1}{2}\Delta_{F_{m,2m,\dots ,km}(\mathbb{C}^n)}$.
    Define
    \begin{equation}\label{eq:blocks-BM}
        W_j(t):=(U_{1j}(t),\dots ,U_{kj}(t))^T,\quad Z_j(t):=U_{(k+1)j}(t)
    \end{equation}
    and note that $w_j(t)=W_j(t)Z_j^{-1}(t)$.
    Itô's formula gives
    \begin{align*}
        dw_j =& -W_jZ_j^{-1}dZ_jZ_j^{-1} 
        +W_jZ_j^{-1}dZ_jZ_j^{-1}dZ_jZ_j^{-1}
        +dW_jZ_j^{-1}
        - dW_jZ_j^{-1}dZ_jZ_j^{-1} \\
        =&\begin{pmatrix}
            I_{n-n_j} & -w_j
        \end{pmatrix}\begin{pmatrix}
            dW_j \\ dZ_j
        \end{pmatrix} Z_j^{-1} -\begin{pmatrix}
             I_{n-n_j} & -w_j
        \end{pmatrix}\begin{pmatrix}
            dW_j \\ dZ_j
        \end{pmatrix} \begin{pmatrix}
            0 & Z_j^{-1}
        \end{pmatrix}\begin{pmatrix}
            dW_j \\ dZ_j
        \end{pmatrix} Z_j^{-1} \\
        =&\begin{pmatrix}
            I_{n-n_j} & -w_j
        \end{pmatrix}\sum_{u=1}^{k+1}\begin{pmatrix}
            W_u \\ Z_u
        \end{pmatrix} dA_{uj}Z_j^{-1} +(2n+1)(W_jZ_j^{-1} -w_j) dt\\
        &-\begin{pmatrix}
            I_{n-n_j} & -w_j
        \end{pmatrix}\sum_{u,v =1}^{k+1}\begin{pmatrix}
            W_u \\ Z_u
        \end{pmatrix} dA_{uj} Z_j^{-1}Z_{v} dA_{vj} Z_j^{-1} .
    \end{align*}
    Note that the bounded variation part is zero, since
    \begin{align*}
        \sum_{u,v =1}^{k+1}\left(\begin{pmatrix}
            W_u \\ Z_u
        \end{pmatrix}dA_{uj} Z_j^{-1}Z_{v} dA_{vj}\right)_{pq}
        =&\sum_{\alpha ,\beta ,\gamma ,\delta =1}^m\sum_{u,v =1}^{k+1}\begin{pmatrix}
            W_u \\ Z_u
        \end{pmatrix}_{p\alpha }(dA_{uj})_{\alpha\beta}
        (Z_{j}^{-1})_{\beta \gamma}( Z_{v}
        )_{\gamma\delta } (dA_{v j})_{\delta q} \\
        =&-\sum_{\alpha ,\beta ,\gamma ,\delta =1}^m\sum_{u,v =1}^{k+1}\begin{pmatrix}
            W_u \\ Z_u
        \end{pmatrix}_{p\alpha }(dA_{uj})_{\alpha\beta}( Z_{j}^{-1})_{\beta \gamma}(Z_{v} )_{\gamma \delta } (d\overline{A}_{jv })_{q\delta} \\
        =&-\sum_{\alpha ,\beta ,\gamma ,\delta =1}^m\sum_{u,v =1}^{k+1}\begin{pmatrix}
            W_u \\ Z_u
        \end{pmatrix}_{p\alpha }( Z_{j}^{-1})_{\beta \gamma}( Z_{v}
        )_{\gamma\delta } \delta_{uj}\delta_{jv }\delta_{\alpha q}\delta_{\delta\beta} \\
        =&-\sum_{\beta ,\gamma =1}^m\begin{pmatrix}
            W_j \\ Z_j
        \end{pmatrix}_{pq }(Z_{j}^{-1})_{\beta\gamma}(Z_j )_{\gamma\beta } \\
        =&-\begin{pmatrix}
            W_j \\ Z_j
        \end{pmatrix}_{pq}\mathrm{Tr}\left( I_m\right) 
        =-m\begin{pmatrix}
            W_j \\ Z_j
        \end{pmatrix}_{pq }.
    \end{align*}
    Summarising
    \begin{equation}\label{eq:fundamental-eq-w}
        dw_j =\begin{pmatrix}
            I_{n-n_j} & -w_j
        \end{pmatrix}\sum_{u=1; u\neq j}^{k+1}\begin{pmatrix}
            W_u \\ Z_u
        \end{pmatrix} dA_{uj}Z_j^{-1}
        =\sum_{u=1;u\neq j}^{k+1}(w_u -w_j) Z_udA_{uj}Z_j^{-1}.
    \end{equation}
    
    Note that this means that $(w(t))_{t\geq 0}$ is a local martingale.
    The quadratic covariances are given by
    \begin{align*}
        (dw_j)_{pq}(d\overline{w}_{\ell})_{rs} 
        =& \sum_{u=1;u\neq j}^{k+1}\sum_{v =1;v\neq\ell}^{k+1}\sum_{\alpha ,\beta ,\gamma ,\delta ,\varepsilon ,\zeta =1}^m(w_u-w_j)_{p\alpha } (Z_u)_{\alpha\beta}d(A_{uj})_{\beta\gamma} (Z_j^{-1})_{\gamma q} \\
        &\quad\qquad\qquad\qquad\qquad\qquad (\overline{w}_v -\overline{w}_{\ell})_{r\delta }(\overline{Z}_{v})_{\delta\varepsilon} d(\overline{A}_{v\ell })_{\varepsilon\zeta} (\overline{Z}_{\ell}^{-1})_{\zeta s} \\
        =&2\delta_{j\ell}\sum_{\substack{u=1 \\ u\neq j}}^{k+1}\sum_{\beta ,\gamma =1}^m((w_u-w_j)(Z_u))_{p\beta} (Z_j^{-1})_{\gamma q}
        ((\overline{w}_u -\overline{w}_j)(\overline{Z}_u))_{r\beta} (\overline{Z}_j^{-1})_{\gamma s} dt \\
        =&2\delta_{j\ell}\sum_{u=1;u\neq j}^{k+1}((w_u-w_j)Z_uZ_u^*(w_u^*-w_j^*))_{pr}((Z_j^*)^{-1}Z_j^{-1})_{sq}\\
        =&2\delta_{j\ell }\sum_{u=1}^{k+1}\left(\begin{pmatrix} I_{n-m} & -w_j\end{pmatrix}\begin{pmatrix}
            W_uW_u^* & W_uZ_u^* \\ Z_uW_u^* & Z_u Z_u^*
        \end{pmatrix}\begin{pmatrix} I_{n-m} \\ -w_j^*\end{pmatrix}\right)_{pr} (Z_j Z_j^*)^{-1}_{s q} dt \\
        =&2\delta_{j\ell }\left(\begin{pmatrix} I_{n-m} & -w_j\end{pmatrix}UU^*\begin{pmatrix} I_{n-m} \\ -w_j^*\end{pmatrix}\right)_{pr} (Z_j Z_j^*)^{-1}_{s q} dt \\
        =&2\delta_{j\ell }\left(\begin{pmatrix} I_{n-m} & -w_j\end{pmatrix}\begin{pmatrix} I_{n-m} \\ -w_j^*\end{pmatrix}\right)_{pr} (Z_j Z_j^*)^{-1}_{s q} dt \\
        =&2\delta_{j\ell }\left( I_{n-m} +w_jw_j^*\right)_{pr} (I_{m} +w_j^*w_j)_{s q} dt ,
    \end{align*}
    where we used the orthogonality of $(U(t))_{t\geq 0}$.
    
    It turns out that the Brownian motion on a partial flag manifold with blocks of equal size is not always conformal.
    Therefore we are also interested in
    \begin{align*}
        (dw_j)_{pq}(dw_{\ell})_{rs} 
        =&\sum_{u=1;u\neq j}^{k+1}\sum_{v=1;v\neq\ell}^{k+1}\sum_{\alpha ,\beta ,\gamma ,\delta ,\varepsilon ,\zeta =1}^m(w_u -w_j)_{p\alpha }(Z_u)_{\alpha\beta} d(A_{uj})_{\beta\gamma} (Z_j^{-1})_{\gamma q}\\
        &\quad\qquad\qquad\qquad\qquad\qquad (w_v-w_{\ell})_{r\delta }(Z_{v})_{\delta\varepsilon} d(A_{v \ell})_{\varepsilon\zeta} (Z_{\ell}^{-1})_{\zeta s} \\
        =&-2\sum_{\alpha ,\beta ,\gamma ,\delta =1}^{m}(w_{\ell}-w_j)_{p\alpha}(Z_{\ell})_{\alpha\beta}(Z_j^{-1})_{\gamma q}(w_j-w_{\ell})_{r\delta}(Z_j)_{\delta\gamma}(Z_{\ell}^{-1})_{\beta s}dt\\
        =&-2(w_{\ell}-w_j)_{ps}(w_j-w_{\ell})_{rq}dt\\
    \end{align*}
    and similarly
    \begin{equation*}
        (d\overline{w}_j)_{pq}(d\overline{w}_{\ell})_{rs} =-2(\overline{w}_{\ell}-\overline{w}_j)_{ps}(\overline{w}_j-\overline{w}_{\ell})_{rq}dt.
    \end{equation*}
    We conclude by applying the complex form of Itô's formula.
\end{proof}

\subsection{Radial motions}\label{subsec:radial-motions}
In this section we will determine the generator of a process related to the process $(J(t))_{t\geq 0}$ of Brownian motion $(w(t))_{t\geq 0}$ on $F_{m,2m,\dots ,km}(\mathbb{C}^n)$ defined in \eqref{eq:radial-process-squared}.

Consider the \emph{squared radial processes}
\begin{equation}\label{eq:the-process-of-interest}
    \Lambda (t):=\left( (I_m+J_1(t))^{-1},\dots ,(I_m +J_{k+1}(t))^{-1}\right)\textrm{ for } t\geq 0.
\end{equation}
Note that $U^*(t)U(t)=I_n$ implies $W_j^*(t)W_j(t)+Z_j^*(t)Z_j(t)=I_m$ and thus $\Lambda_j(t)=Z_j(t)Z_j^*(t)$, with $(W_j(t))_{t\geq 0}$ and $(Z_j(t))_{t\geq 0}$ defined as in \eqref{eq:blocks-BM}.
From this we can easily see that $(\Lambda (t))_{t\geq 0}$ is stochastic process in the \emph{simplex of Hermitian matrices}
\begin{equation*}
    \mathcal{T}^m_{k+1} :=\left\{ (\Lambda_1,\dots ,\Lambda_{k+1})\in M_{m\times m}(\mathbb{C})^{k+1}\Bigm| \sum_{j=1}^{k+1}\Lambda_j =I_m, \Lambda_j\geq 0 ,\Lambda_j^*=\Lambda_j\textrm{ for } 1\leq j\leq k+1\right\} .
\end{equation*}
Furthermore, the process \eqref{eq:the-process-of-interest} turns out to be a diffusion as shown in theorem \ref{thm:generator-process-of-interest} below.

\begin{theorem}\label{thm:generator-process-of-interest}
    Let $(w_1 (t),\dots ,w_{k}(t))_{t\geq 0}$ be a Brownian motion on the partial flag manifold $F_{m,2m,\dots ,km}(\mathbb{C}^n)$.
    The process $(\Lambda (t))_{t\geq 0}$ given by \eqref{eq:the-process-of-interest} satisfies the stochastic differential equation
    \begin{equation*}
        d\Lambda_j (t) = \Lambda_j^{\frac{1}{2}}(t)\sum_{\ell =1;\ell\neq j}^{k+1}(d\gamma_{\ell j})^*(t) \Lambda_{\ell}^{\frac{1}{2}}(t) 
        +\sum_{\ell =1;\ell\neq j}^{k+1}\Lambda_{\ell}^{\frac{1}{2}}(t) d\gamma_{\ell j}(t)\Lambda_j^{\frac{1}{2}}(t) +2(mI_m-n\Lambda_j(t) )dt,
    \end{equation*}
    where $(\gamma(t))_{t\geq 0}$ is a Brownian motion on $\mathfrak{u}(n)$.
    Consequently, $(\Lambda (t))_{t\geq 0}$ is a multi-variate matrix-valued diffusion on the simplex $\mathcal{T}_{k+1}^m$ with generator $2\hat{\mathcal{G}}_{m/2,\dots ,m/2}^m$, where
    \begin{align*}
        \hat{\mathcal{G}}_{m/2,\dots ,m/2}^m=&\sum_{j=1}^{k+1}\sum_{\alpha ,\beta ,\gamma ,\delta =1}^m (I_m-\Lambda_j)_{\alpha\delta}\Lambda_{j,\gamma\beta}\frac{\partial^2}{\partial\Lambda_{j,\alpha\beta}\partial\Lambda_{j,\gamma\delta}} \\
        &-\frac{1}{2}\sum_{1\leq j\neq \ell\leq k+1}\sum_{\alpha ,\beta ,\gamma ,\delta =1}^m(\Lambda_{j,\gamma\beta}\Lambda_{\ell ,\alpha\delta} +\Lambda_{j,\alpha\delta}\Lambda_{\ell ,\gamma\beta})\frac{\partial^2}{\partial\Lambda_{j,\alpha\beta}\partial\Lambda_{\ell ,\gamma\delta}} \\
        &+\sum_{j=1}^{k+1}\sum_{\alpha ,\beta =1}^{m} (mI_m-n\Lambda_{j})_{\alpha\beta}\frac{\partial}{\partial\Lambda_{j,\alpha\beta}}.
    \end{align*}
\end{theorem}
\begin{proof}
    Since $Z_{j}(t)Z_{j}^*(t)=\Lambda_j(t)$, there exists a $Q_j\in U(m)$ such that $Z_j(t)= \Lambda^{\frac{1}{2}}_j(t)Q_j(t)$.
    Equation \eqref{eq:fundamental-eq-w} gives
    \begin{equation*}
        dw_j =\sum_{\ell =1;\ell\neq j}^{k+1}(w_i-w_j)\Lambda^{\frac{1}{2}}_{\ell}Q_{\ell}dA_{\ell j}Q^*_j\Lambda^{-\frac{1}{2}}_j.
    \end{equation*}
    Define
    \begin{equation*}
        d\gamma_{\ell j} := Q_{\ell}dA_{\ell j}Q_j^*\textrm{ for }1\leq\ell\neq j\leq k+1 .
    \end{equation*}
    Note that $(\gamma_{\ell j})^* =-\gamma_{j\ell}$ and that for $1\leq j\neq\ell\leq k+1$ and $1\leq r\neq s\leq k+1$ we have
    \begin{align*}
        (d\gamma_{\ell j})_{pq}(d\overline{\gamma}_{rs})_{p'q'} =&\sum_{\alpha ,\beta ,\alpha' ,\beta' =1}^m Q_{\ell, p\alpha}(dA_{\ell j})_{\alpha\beta}(Q_{j}^*)_{\beta q}\overline{Q}_{s,p'\alpha'}(d\overline{A}_{r s})_{\alpha'\beta'}(\overline{Q}_{r}^*)_{\beta' q'}\\
        =&2\delta_{r\ell}\delta_{js}\sum_{\alpha ,\beta ,\gamma =1}^m Q_{\ell, p\alpha}(Q_{j}^*)_{\beta q}\overline{Q}_{j,p'\alpha}(\overline{Q}_{\ell}^*)_{\beta q'} dt
        =2\delta_{r\ell}\delta_{js}\delta_{pp'}\delta_{q'q} dt,
    \end{align*}
    which shows that $(\gamma (t))_{t\geq 0}$ is a Brownian motion on $\mathfrak{u}(n)$ by the Lévy characterisation theorem, if we define $(\gamma_{jj}(t))_{t\geq 0}$ correctly for $1\leq j\leq k+1$.
    Note that these will not appear in the equation below.
    In particular,
    \begin{equation}\label{eq:derivative-Brownian-motion}
        dw_j =\sum_{\ell =1;\ell\neq j}^{k+1}(w_{\ell}-w_j)\Lambda^{\frac{1}{2}}_{\ell}d\gamma_{\ell j}\Lambda^{-\frac{1}{2}}_j.
    \end{equation}

    Furthermore for matrices $M\in\mathbb{C}^{m\times m}$ we have
    \begin{align*}
        (dw_j^*Mdw_{\ell})_{pq}=&\sum_{\alpha ,\beta =1}^m(dw_j^*)_{p\alpha} M_{\alpha\beta}(dw_j)_{\beta q}\\
        =&\sum_{\substack{u,v =1\\ u\neq j,v\neq\ell}}^{k+1}\sum_{\alpha ,\beta ,\gamma ,\delta =1}^m\Lambda^{-\frac{1}{2}}_{j,p\alpha}(d\overline{\gamma}_{uj })_{\beta\alpha}(\Lambda^{\frac{1}{2}}_{u}(w_u-w_j)^*M(w_v-w_{\ell})\Lambda^{\frac{1}{2}}_{v})_{\beta\gamma}(d\gamma_{v\ell })_{\gamma\delta}\Lambda^{-\frac{1}{2}}_{\ell ,\delta q} \\
        =&2\delta_{j\ell}\sum_{u =1;u\neq j}^{k+1}\sum_{\alpha ,\beta =1}^m\Lambda^{-\frac{1}{2}}_{j,p\alpha}(\Lambda^{ \frac{1}{2}}_{u}(w_u-w_j)^*M(w_u-w_{j})\Lambda^{\frac{1}{2}}_u)_{\beta\beta}\Lambda^{-\frac{1}{2}}_{j,\alpha q} dt\\
        =&2\delta_{j\ell}\sum_{u =1;u\neq j}^{k+1} \mathrm{Tr}((w_u-w_j)^*M(w_u-w_{j})\Lambda_u)\Lambda^{-1}_{j,pq} dt
    \end{align*}
    and similarly
    \begin{equation*}
        dw_jMdw_{\ell}^* =2\delta_{j\ell}\sum_{u=1;u\neq j}^{k+1}(w_u-w_j)\Lambda_u^{\frac{1}{2}}(w_u^*-w_j^*)\mathrm{Tr}(M\Lambda_j^{-1}) dt.
    \end{equation*}
    The other covariances are zero.

    Itô's formula gives
    \begin{align*}
        d\Lambda_j =&d(I_m+w^*_jw_j)^{-1}
        =-(I_m+w_j^*w_j)^{-1}(dw^*_jw_j +w_j^*dw_j+dw_j^*dw_j)(I_m+w_j^*w_j)^{-1} \\
        &+(I_m+w_j^*w_j)^{-1}(dw^*_jw_j +w_j^*dw_j)(I_m+w_j^*w_j)^{-1}(dw^*_jw_j +w_j^*dw_j)(I_m+w_j^*w_j)^{-1} \\
        =&-\Lambda_j(dw^*_jw_j +w_j^*dw_j+dw_j^*dw_j)\Lambda_j +\Lambda_j(dw^*_jw_j +w_j^*dw_j)\Lambda_j(dw^*_jw_j +w_j^*dw_j)\Lambda_j .
    \end{align*}
    Together with the relation $w_j^*w_{\ell}=-I_m$ for $1\leq j\neq\ell\leq k+1$ this gives
    \begin{align*}
        d\Lambda_j=&-\sum_{\ell =1;\ell\neq j}^{k+1}\left(\Lambda_{j}^{\frac{1}{2}}d\gamma_{\ell j}^*\Lambda_{\ell}^{\frac{1}{2}} +\Lambda_{\ell}^{\frac{1}{2}}d\gamma_{\ell j}\Lambda_{j}^{\frac{1}{2}}\right) 
        -2\sum_{u=1;u\neq j}^{k+1}\Lambda_j\mathrm{Tr}((w_u^* -w_j^*)(w_u -w_j)\Lambda_u) dt\\
        &+2\Lambda_j\sum_{u=1;u\neq j}^{k+1}\mathrm{Tr}((w_u^*-w_j^*)w_j\Lambda_jw_j^*(w_u-w_j)\Lambda_u ) dt\\
        &+2\sum_{u=1;u\neq j}^{k+1}\Lambda_jw_j^*(w_u-w_j)\Lambda_u(w_u^* -w_j^*)w_j\Lambda_j\mathrm{Tr}(\Lambda_j^{-1}\Lambda_j) dt\\
        =&-\sum_{\ell =1;\ell\neq j}^{k+1}\left(\Lambda_{j}^{\frac{1}{2}}d\gamma_{\ell j}^*\Lambda_{\ell}^{\frac{1}{2}} +\Lambda_{\ell}^{\frac{1}{2}}d\gamma_{\ell j}\Lambda_{j}^{\frac{1}{2}}\right) 
        -2\mathrm{Tr}(kI_m+\Lambda_j^{-1}(I_m-\Lambda_j) )\Lambda_jdt \\
        &+2\mathrm{Tr}(\Lambda^{-1}_j(I_m-\Lambda_j))\Lambda_j dt
        + 2m(I_m-\Lambda_j)dt\\
        =&\Lambda_j^{\frac{1}{2}}\sum_{\ell =1;\ell\neq j}^{k+1}d\gamma^*_{\ell j} \Lambda_{\ell}^{\frac{1}{2}} +\sum_{\ell =1;\ell\neq j}^{k+1}\Lambda_{\ell}^{\frac{1}{2}} d\gamma_{\ell j}\Lambda_j^{\frac{1}{2}} +2(mI_m-n\Lambda_j )dt .
    \end{align*}

    The expression for twice the generator $2\hat{\mathcal{G}}_{m/2,\dots ,m/2}^{m}$ follows by a simple calculation.
    More precisely $(d\Lambda_j)_{\alpha\beta}(d\Lambda_{\ell})_{\gamma\delta}$ gives the coefficient of the $\frac{\partial^2}{\partial\Lambda_{j,\alpha\beta}\partial\Lambda_{\ell ,\gamma\delta}}$ term times $dt$ and the $(\alpha ,\beta)$-component of the bounded variation part gives the $\frac{\partial}{\partial\Lambda_{j,\alpha\beta}}$ term. 
    Alternatively, one can determine how the Laplace--Beltrami operator $\Delta_{F_{m,2m,\dots ,km}(\mathbb{C})}$ acts on functions depending only on
    \begin{equation}\label{eq:coordinates-of-interest}
        \Lambda :=\begin{pmatrix} (I_m+w^*_1w_1)^{-1} &\dots & (I_m+w_{k+1}^*w_{k+1})^{-1}\end{pmatrix}
    \end{equation}
    using lemma \ref{lemma:derivatives-matrices}.
\end{proof}

In analogy with \cite[Theorem 3.5.]{baudoin2025fullflag} and \cite[Proposition 8.1.1.]{Baudoin2024-is} we will introduce the following family of diffusions.
For other parameters these will turn out to be relevant in theorem \ref{thm:Laplace-transform-area-functional-equal-size}.

\begin{definition}\label{def:matrix-Jacobi-process-simplex}
    Let $\kappa =(\kappa_1,\dots ,\kappa_{k+1})$ be a multi-index such that $\kappa_j >m/2 -1$ for $1\leq j\leq k+1$.
    The \emph{Jacobi operator} on $\mathcal{T}^m_{k+1}$ of index $\kappa$ is defined by
    \begin{align*}
            \mathcal{G}_{\kappa}^m :=&\frac{1}{2}\sum_{j=1}^{k+1}\sum_{\alpha ,\beta ,\gamma ,\delta =1}^m ((I_m-\Lambda_j)_{\alpha\delta}\Lambda_{j,\gamma\beta} +(I_m-\Lambda_j)_{\gamma\beta}\Lambda_{j,\alpha\delta})\frac{\partial^2}{\partial\Lambda_{j,\alpha\beta}\partial\Lambda_{j,\gamma\delta}} \\ 
            &-\frac{1}{2}\sum_{1\leq j\neq\ell\leq k+1}\sum_{\alpha ,\beta ,\gamma ,\delta =1}^m (\Lambda_{j,\alpha\delta}\Lambda_{\ell ,\gamma\beta} +\Lambda_{j,\gamma\beta}\Lambda_{\ell ,\alpha\delta})\frac{\partial^2}{\partial\Lambda_{j,\alpha\beta}\partial\Lambda_{\ell ,\gamma\delta}}\\
            & +\sum_{j=1}^{k+1}\sum_{\alpha ,\beta =1}^m\left(\left(\kappa_j +\frac{m}{2}\right) \delta_{\alpha\beta}-\left(|\kappa |+\frac{n}{2}\right)\Lambda_{j,\alpha\beta}\right)\frac{\partial}{\partial\Lambda_{j,\alpha\beta}} ,
    \end{align*}
    where we used the notation $|\kappa |:=\kappa_1 +\dots +\kappa_{k+1}$ and $n=(k+1)m$.
    A matrix diffusion $\Lambda =(\Lambda_1 ,\dots ,\Lambda_{k})$ with generator $\mathcal{G}_{\kappa}^m$ is called a \emph{Jacobi process on the simplex of Hermitian matrices} of index $\kappa$ or more simply a \emph{Hermitian Jacobi process on the simplex}.
\end{definition}

\begin{remark}\label{rmk:marginal-ditribution}
    Let $(\Lambda (t))_{t\geq 0}=(\Lambda_1 (t),\dots ,\Lambda_{k+1}(t))_{t\geq 0}$ be a Jacobi process on the simplex of Hermitian matrices.
    It is straightforward to show that, for $1\leq j\leq k+1$, $\mathcal{G}^m_{\kappa}$ acts on functions depending only on $\Lambda_j$ as
    \begin{equation*}
        \scalemath{0.95}{\sum_{\alpha ,\beta ,\gamma ,\delta =1}^m (I_m-\Lambda_j)_{\alpha\delta}\Lambda_{j,\gamma\beta} \frac{\partial^2}{\partial\Lambda_{j,\alpha\beta}\partial\Lambda_{j,\gamma\delta}} 
        +\sum_{\alpha ,\beta =1}^m\left(\left(\kappa_j +\frac{m}{2}\right) \delta_{\alpha\beta}-\left(|\kappa |+\frac{n}{2}\right)\Lambda_{j,\alpha\beta}\right)\frac{\partial}{\partial\Lambda_{j,\alpha\beta}}} .
    \end{equation*}
    This is precisely the generator of a Hermitian Jacobi process of index $(\kappa_j -m/2 ,|\kappa|-\kappa_j-m/2)$.
    The marginal processes $(\Lambda_j(t))_{t\geq 0}$ are therefore Hermitian Jacobi processes for $1\leq j\leq k+1$.
\end{remark}

\begin{remark}
    The calculation in theorem \ref{thm:Laplace-Beltrami-operator-flags} can also be carried out for general partial flag manifolds without problem.
    However, when calculating how the corresponding Laplace--Beltrami operator acts on functions depending only on $\Lambda =(\Lambda_1,\dots ,\Lambda_{k+1})$, defined analogues to \eqref{eq:coordinates-of-interest}, one sees that the corresponding expression reduces to an expression in terms of $\Lambda$ itself precisely when it has blocks of equal size.
\end{remark}

\section{Stochastic area functionals and horizontal Brownian motions}\label{chap:stochastic-area}
Recall the squared radial process $(\Lambda (t))_{t\geq 0}$ defined in \eqref{eq:the-process-of-interest}.
Because of the embedding
\begin{equation*}
    F_{m,2m,\dots ,km}(\mathbb{C}^n)\hookrightarrow (Gr_{n,m}(\mathbb{C}))^{k+1}
\end{equation*}
it makes sense to define the \emph{stochastic area processes} by $\mathfrak{a}(t) =(\mathfrak{a}_1(t),\dots ,\mathfrak{a}_{k+1}(t))$ with
\begin{equation}\label{eq:stochastic-area-process}
    d\mathfrak{a}_j(t) =\frac{1}{2i}\mathrm{Tr}\left(\Lambda_j (t)(dw_j^*(t) w_j(t) -w_j^*(t)dw_j(t))\right)
\end{equation}
the area process on the Grassmannian $Gr_{n,m}(\mathbb{C})$, where this stochastic differential equation is understood in the Stratonovich sense, or equivalently in the Itô sense, see \cite[\textsection 8.3.2]{Baudoin2024-is}.

\begin{remark}
    The space $F_{m,2m,\dots ,km}(\mathbb{C}^n)$ is a Kähler manifold, moreover it is even a projective variety, since it is embedded in $Gr_{n,m}(\mathbb{C})\times\dots\times Gr_{n,m}(\mathbb{C})$. Its Kähler form is given on $\mathcal{O}_m$ by
    \begin{equation*}
        \omega =\sum_{j=1}^{k+1}d\mathfrak{a}_j
    \end{equation*}
    with $d\mathfrak{a}_j :=\frac{i}{2}\mathrm{Tr}((1+w_j^*w_j)^{-1}(dw_j^*w_j-w_j^*dw_j))$ the area form on $Gr_{n,m}(\mathbb{C})$.
    The area forms corresponding to \eqref{eq:stochastic-area-process} can thus be seen as a canonical decomposition of the Kähler form.
\end{remark}

From theorem \ref{thm:generator-process-of-interest} we can show that the stochastic area process $(\mathfrak{a}(t))_{t\geq 0}$ is part of a diffusion.

\begin{corollary}\label{cor:main-diffusion-m}
    The stochastic process $(\Lambda (t),\mathfrak{a}(t))_{t\geq 0}$ is a diffusion with generator
    \begin{align*}
        L:=&2\sum_{j=1}^{k+1}\sum_{\alpha ,\beta ,\gamma ,\delta =1}^m (I_m-\Lambda_j)_{\alpha\delta}\Lambda_{j,\gamma\beta}\frac{\partial^2}{\partial\Lambda_{j,\gamma\delta}\partial\Lambda_{j,\alpha\beta}} 
        +2\sum_{j=1}^{k+1}\sum_{\alpha ,\beta =1}^m (mI_m-n\Lambda_{j})_{\alpha\beta}\frac{\partial}{\partial\Lambda_{j,\alpha\beta}}\\
        &-\sum_{1\leq j\neq \ell\leq k+1}^{k+1}\sum_{\alpha ,\beta ,\gamma ,\delta =1}^m(\Lambda_{j,\gamma\beta}\Lambda_{\ell ,\alpha\delta} +\Lambda_{j,\alpha\delta}\Lambda_{\ell ,\gamma\beta})\frac{\partial^2}{\partial\Lambda_{j,\alpha\beta}\partial\Lambda_{\ell ,\gamma\delta}}\\
        &+\sum_{j=1}^{k+1}(\mathrm{Tr}(\Lambda_j^{-1}) -m)\frac{\partial^2}{\partial\mathfrak{a}_j^2}
        +m\sum_{1\leq j\neq\ell\leq k+1}\frac{\partial^2}{\partial\mathfrak{a}_j\partial\mathfrak{a}_{\ell}}.
    \end{align*}
\end{corollary}
\begin{proof}
    All terms except those of the form $\frac{\partial^2}{\partial\mathfrak{a}_j\partial\Lambda_m}$ and $\frac{\partial^2}{\partial\mathfrak{a}_j\partial\mathfrak{a}_m}$ follow straightforwardly from theorem \ref{thm:generator-process-of-interest}.
    For these cross-terms we calculate, using $w_j^*w_{\ell} =-I_{m}$ for $1\leq j\neq\ell\leq k$, that
    \begin{align*}
        d\mathfrak{a}_{j}d\Lambda_{\ell ,pq} =&\frac{i}{2}\sum_{\alpha ,\beta ,\delta ,\zeta=1}^m\sum_{\gamma ,\varepsilon =1}^{n-m} \Lambda_{j,\alpha\beta}
        (d\overline{w}_{j ,\gamma\beta }w_{j,\gamma\alpha} -\overline{w}_{j,\gamma\beta }dw_{j,\gamma\alpha })
        \Lambda_{\ell,p\delta }
        (d\overline{w}_{\ell ,\varepsilon\delta}w_{\ell ,\varepsilon\zeta} +\overline{w}_{\ell ,\varepsilon\delta}dw_{\ell ,\varepsilon\zeta})\Lambda_{\ell,\zeta q} \\
        =&-\frac{i}{2}\sum_{\alpha ,\beta ,\delta ,\zeta =1}^m \sum_{\gamma ,\varepsilon =1}^{n-m}\Lambda_{j,\alpha\beta }\Lambda_{\ell ,p\delta } \Lambda_{\ell ,\zeta q}  (w_{j,\gamma\alpha}w_{\ell ,\varepsilon\zeta}d\overline{w}_{j,\gamma\beta }d\overline{w}_{\ell ,\varepsilon\delta} +w_{j,\gamma\alpha }\overline{w}_{\ell ,\varepsilon\delta} \\
        &\quad d\overline{w}_{j,\gamma\beta }dw_{\ell ,\varepsilon\zeta} 
        -\overline{w}_{j,\gamma\beta }w_{\ell ,\varepsilon\zeta } dw_{j,\gamma\alpha }d\overline{w}_{\ell ,\varepsilon\delta }
        -\overline{w}_{j ,\gamma\beta }\overline{w}_{\ell ,\varepsilon\delta}dw_{j,\gamma\alpha }dw_{\ell ,\varepsilon\zeta} )\\
        =&-i\sum_{\alpha ,\beta ,\delta ,\zeta =1}^m \sum_{\gamma ,\varepsilon =1}^{n-m} \Lambda_{j,\alpha\beta }\Lambda_{\ell ,p\delta } \Lambda_{\ell ,\zeta q} ( -w_{j,\gamma\alpha}w_{\ell ,\varepsilon\zeta}
        (\overline{w}_{\ell} -\overline{w}_j)_{\gamma\delta}(\overline{w}_j -\overline{w}_{\ell})_{\varepsilon\beta} \\
        & +\delta_{j\ell }w_{j,\gamma\alpha}\overline{w}_{j,\varepsilon\delta} (I_{n-m} +w_jw_j^*)_{\varepsilon\gamma}\Lambda^{-1}_{j,\beta\zeta}
        -\delta_{j\ell }\overline{w}_{j,\gamma\beta }w_{j,\varepsilon\zeta}(I_{n-m} +w_jw_j^*)_{\gamma\varepsilon }\Lambda^{-1}_{j,\delta\alpha}\\
        & -\overline{w}_{j,\gamma\beta }\overline{w}_{\ell ,\varepsilon\delta }(w_{\ell} -w_j)_{\gamma\zeta}(w_{j} -w_{\ell})_{\varepsilon\alpha } )\\
        =&i(\Lambda_{\ell}(w_{\ell}^* -w_j^*)w_j\Lambda_j (w_j^* -w_{\ell}^*)w_{\ell}\Lambda_{\ell })_{pq}
        -i\delta_{j\ell}(\Lambda_{\ell}w_j^*(I_{n-m} +w_jw_j^*)w_j\Lambda_{\ell})_{pq} \\
        &+i\delta_{j\ell }(\Lambda_{\ell}w_j^*(I_{n-m}+w_jw_j^*)w_j\Lambda_{\ell})_{pq}
        -i(\Lambda_{\ell} w_{\ell}^* (w_j-w_{\ell})\Lambda_jw_j^*(w_{\ell} -w_j)\Lambda_{\ell})_{pq} \\
        =&i(\Lambda_{\ell}\Lambda_{j}^{-1}\Lambda_{\ell})_{pq} -i(\Lambda_{\ell}\Lambda_{j}^{-1}\Lambda_{\ell})_{pq} =0 .
    \end{align*}

    For the other terms we calculate
    \begin{align*}
        d\mathfrak{a}_jd\mathfrak{a}_{\ell} =&-\frac{1}{4}\sum_{p,q,\alpha ,\gamma =1}^m\sum_{\beta ,\delta =1}^{n-m} \Lambda_{j,p\alpha}(d\overline{w}_{j,\beta\alpha }w_{j,\beta p} -\overline{w}_{j,\beta\alpha }dw_{j,\beta p}) \\
        &\qquad\qquad\qquad\qquad\Lambda_{\ell ,q\gamma}(d\overline{w}_{\ell,\delta\gamma }w_{\ell ,\delta q} -\overline{w}_{\ell,\delta\gamma }dw_{\ell ,\delta q})\\
        =&-\frac{1}{4}\sum_{p,q,\alpha ,\gamma =1}^m\sum_{\beta ,\delta =1}^{n-m} \Lambda_{j,p\alpha}\Lambda_{\ell ,q\gamma} (d\overline{w}_{j,\beta\alpha }w_{j,\beta p} d\overline{w}_{\ell ,\delta\gamma }w_{\ell ,\delta q} 
        -d\overline{w}_{j,\beta\alpha }w_{j,\beta p} 
        \overline{w}_{\ell ,\delta\gamma }dw_{\ell ,\delta q}\\
        &\qquad\qquad\qquad\qquad\qquad\quad -\overline{w}_{j,\beta\alpha }dw_{j,\beta p}d\overline{w}_{\ell ,\delta\gamma }w_{\ell ,\delta q} 
        +\overline{w}_{j,\beta\alpha }dw_{j,\beta p}\overline{w}_{\ell ,\delta\gamma }dw_{\ell ,\delta q}) \\
        =&\frac{1}{2}\sum_{p,q,\alpha ,\gamma =1}^m\sum_{\beta ,\delta =1}^{n-m} \Lambda_{j,p\alpha}\Lambda_{\ell ,q\gamma} \big( w_{j,\beta p}w_{\ell ,\delta q}\left(\overline{w}_{\ell} -\overline{w}_j\right)_{\beta\gamma}
        \left( \overline{w}_j-\overline{w}_{\ell }\right)_{\delta\alpha}\\
        &+\delta_{j\ell }w_{j,\beta p}\overline{w}_{j,\delta\gamma}\left( I_{n-m} +w_jw_j^*\right)_{\delta\beta} \Lambda_{j,\alpha q}^{-1}
        +\delta_{j\ell }\overline{w}_{j,\beta\alpha}w_{j,\delta q}\left( I_{n-m} +w_jw_j^*\right)_{\beta\delta} \Lambda_{j,\gamma p}^{-1}\\
        &+\overline{w}_{j,\beta\alpha}\overline{w}_{\ell ,\delta\gamma }\left( w_{\ell }-w_j\right)_{\beta q}
        \left( w_j-w_{\ell}\right)_{\delta p}\big) dt.
    \end{align*}
    Simplification gives the desired generator.
\end{proof}

\subsection{Horizontal Brownian motion}
Recall the Riemannian submersion $p:\mathcal{D}_m\rightarrow\mathcal{O}_m$ defined in \eqref{eq:Riemannian-submersion-coordinates} and the Laplace--Beltrami operator $\Delta_{F_{m,2m,\dots ,km}(\mathbb{C}^n)}$ from theorem \ref{thm:Laplace-Beltrami-operator-flags}.
The operator $\Delta_{\mathcal{H}_F}$ satisfying
\begin{equation*}
    \Delta_{F_{m,2m,\dots ,km}(\mathbb{C}^n)} (f\circ p)=(\Delta_{\mathcal{H}_F}f)\circ p\textrm{ for } f\in C^{\infty}(\mathcal{D})
\end{equation*}
is called the horizontal Laplacian of the fibration $\mathbf{U}(m)^{k+1}\rightarrow\mathbf{U}(n)\rightarrow F_{m,2m,\dots ,km}(\mathbb{C}^n)$.

\begin{definition}
    A \emph{horizontal Brownian motion} with respect to the fibration $\mathbf{U}(m)^{k+1}\rightarrow\mathbf{U}(n)\rightarrow F_{m,2m,\dots ,km}(\mathbb{C}^n)$ is a diffusion on $\mathcal{D}$ with generator $\frac{1}{2}\Delta_{\mathcal{H}_F}$.
\end{definition}

\begin{theorem}\label{thm:horizontal-MB-m}
    Let $(w(t))_{t\geq 0}$ be a Brownian motion on $F_{m,2m,\dots ,km}(\mathbb{C}^n)$ and let $(\mathfrak{a}(t))_{t\geq 0}$ be its stochastic area process.
    The process
    \begin{equation*}
        X(t) :=\scalemath{0.88}{\begin{pmatrix}
            e^{i\frac{\mathfrak{a}_1(t)}{m}}w_{11}(t)(I_m +w_1^*(t)w_1(t))^{-\frac{1}{2}}\Theta_1 (t) & \dots & e^{i\frac{\mathfrak{a}_{k+1}(t)}{m}}w_{1(k+1)}(t)(I_m +w_{k+1}^*(t)w_{k+1}(t))^{-\frac{1}{2}}\Theta_{k+1} (t) \\
            \vdots & \ddots & \vdots \\
            e^{i\frac{\mathfrak{a}_1(t)}{m}}w_{k1}(t)(I_m +w_1^*(t)w_1(t))^{-\frac{1}{2}}\Theta_1 (t) & \dots & e^{i\frac{\mathfrak{a}_{k+1}(t)}{m}}w_{k(k+1)}(t)(I_m +w_{k+1}^*(t)w_{k+1}(t))^{-\frac{1}{2}}\Theta_{k+1} (t) \\
            e^{i\frac{\mathfrak{a}_1(t)}{m}}(I_m +w_1^*(t)w_1(t))^{-\frac{1}{2}}\Theta_1 (t) & \dots & e^{i\frac{\mathfrak{a}_{k+1}(t)}{m}}(I_m +w_{k+1}^*(t)w_{k+1}(t))^{-\frac{1}{2}}\Theta_{k+1} (t)
        \end{pmatrix}}
    \end{equation*}
    is a horizontal Brownian motion on $\mathbf{U}(n)$ with respect to the fibration $\mathbf{U}(m)^{k+1}\rightarrow\mathbf{U}(n)\rightarrow F_{m,2m,\dots ,km}(\mathbb{C}^n)$, where $(\Theta (t))_{t\geq 0}$ is a $\mathbf{SU}(m)^{k+1}$-valued process satisfying the system of Stratonovich stochastic differential equations
    \begin{equation*}
        d\Theta_j(t) =\circ id\eta_j(t)\Theta_j(t) +\tfrac{i}{m}\Theta_j(t)\circ d\mathfrak{a}_j(t)
    \end{equation*}
    with
    \begin{equation*}
        d\eta_j (t):=\frac{1}{2i}\left(\Lambda^{\frac{1}{2}}_j(t)(\circ dw^*_j(t)w_j(t)-w^*_j(t)\circ dw_j(t))\Lambda_{j}^{-\frac{1}{2}}(t) -\Lambda^{\frac{1}{2}}_j(t)\circ d\Lambda_j^{-\frac{1}{2}}(t) +\circ d\Lambda_j^{-\frac{1}{2}}(t)\Lambda_j^{\frac{1}{2}}(t)\right) .
    \end{equation*}
\end{theorem}
\begin{proof}
    From \cite[Lemma 8.2.1.]{Baudoin2024-is} it can easily be checked that the connection form of the fibration $\mathbf{U}(m)^{k+1}\rightarrow\mathbf{U}(n)\rightarrow F_{m,2m,\dots ,km}(\mathbb{C})$ is given by $\omega =(\omega_1,\dots ,\omega_{k+1})$, where
    \begin{equation*}
        \omega_j :=\frac{1}{2i}(w_j^*dw_j-dw^*_jw_j),
    \end{equation*} 
    see also \cite[Lemma 4.1.]{baudoin2025fullflag}.
    Similarly to the proof of \cite[Theorem 8.2.2.]{Baudoin2024-is} one sees that a horizontal Brownian motion on $\mathbf{U}(n)$ is given by
    \begin{equation*}
        \begin{pmatrix}
            w_{11}(t)(I_m +w_1^*(t)w_1(t))^{-\frac{1}{2}}\tilde{\Theta}_{1} (t) & \dots &  w_{1(k+1)}(t)(I_m +w_{k+1}^*(t)w_{k+1}(t))^{-\frac{1}{2}}\tilde{\Theta}_{k+1} (t) \\
            \vdots & \ddots & \vdots \\
            w_{k1}(t)(I_m +w_1^*(t)w_1(t))^{-\frac{1}{2}}\tilde{\Theta }_1 (t) & \dots & w_{k(k+1)}(t)(I_m +w_{k+1}^*(t)w_{k+1}(t))^{-\frac{1}{2}}\tilde{\Theta}_{k+1}(t) \\
            (I_m +w_1^*(t)w_1(t))^{-\frac{1}{2}}\tilde{\Theta}_1 (t) & \dots & (I_m +w_{k+1}^*(t)w_{k+1}(t))^{-\frac{1}{2}}\tilde{\Theta}_{k+1}(t)
        \end{pmatrix} ,
    \end{equation*}
    where $(\tilde{\Theta}(t))_{t\geq 0}$ solves the Stratonovich stochastic differential equation
    \begin{equation*}
        d\tilde{\Theta}_j(t) =\circ id\eta_j (t)\tilde{\Theta}_j (t)\textrm{ for } 1\leq j\leq k+1.
    \end{equation*}
    
    We compute using the Itô's formula in Stratonovich form
    \begin{equation*}
        \circ d\det (\tilde{\Theta}_j)=\det (\tilde{\Theta}_j)\mathrm{Tr}(\tilde{\Theta}_j^{-1}\circ d\tilde{\Theta}_j)
        =i\det (\tilde{\Theta}_j)\mathrm{Tr}(\circ d\eta_j)
        =\det (\tilde{\Theta}_j)\circ id\mathfrak{a}_j .
    \end{equation*}
    This means in particular that $\det (\tilde{\Theta}_j (t)) =e^{i\mathfrak{a}_j(t)}$.
    Now we define the $\mathbf{SU}(m)$-valued processes $\Theta_j(t) :=e^{-i\mathfrak{a}_j(t)/m}\tilde{\Theta}_j(t)$ for $1\leq j\leq k+1$.
    By the above calculations
    \begin{equation*}
        \circ d\Theta_j =\frac{\circ d\tilde{\Theta}_j}{\det (\tilde{\Theta}_j )^{1/m}} +\frac{1}{m}\frac{\tilde{\Theta}_j\circ d\det (\tilde{\Theta}_j)}{\det (\tilde{\Theta}_j)^{1/m+1}}
        =\circ id\eta_j\Theta_j +\frac{i}{m}\Theta_j\circ d\mathfrak{a}_j
    \end{equation*}
    for $1\leq j\leq k+1$
\end{proof}

From this we can express a Brownian motion on the canonical torus bundle in terms of a Brownian motion on $F_{m,2m,\dots ,km}(\mathbb{C}^n)$ and their area processes.

\begin{corollary}\label{cor:Brownian-motion-torus-bundle}
    Let $(w(t))_{t\geq 0}$ be a Brownian motion on $F_{m,2m,\dots ,km}(\mathbb{C}^n)$ and let $(\mathfrak{a}(t))_{t\geq 0}$ be its stochastic area process.
    The projection of the process
    \begin{equation*}
        \tilde{X}(t) :=\scalemath{0.85}{\begin{pmatrix}
            e^{i\frac{\mathfrak{a}_1(t)}{m} +i\beta_1(t)}w_{11}(t)(I_m +w_1^*(t)w_1(t))^{-\frac{1}{2}} & \dots & e^{i\frac{\mathfrak{a}_{k+1}(t)}{m} +i\beta_{k+1}(t)}w_{1(k+1)}(t)(I_m +w_{k+1}^*(t)w_{k+1}(t))^{-\frac{1}{2}} \\
            \vdots & \ddots & \vdots \\
            e^{i\frac{\mathfrak{a}_1(t)}{m}+i\beta_1(t)}w_{k1}(t)(I_m +w_1^*(t)w_1(t))^{-\frac{1}{2}} & \dots & e^{i\frac{\mathfrak{a}_{k+1}(t)}{m} +i\beta_{k+1}(t)}w_{k(k+1)}(t)(I_m +w_{k+1}^*(t)w_{k+1}(t))^{-\frac{1}{2}} \\
            e^{i\frac{\mathfrak{a}_1(t)}{m} +i\beta_1(t)}(I_m +w_1^*(t)w_1(t))^{-\frac{1}{2}} & \dots & e^{i\frac{\mathfrak{a}_{k+1}(t)}{m} +i\beta_{k+1}(t)}(I_m +w_{k+1}^*(t)w_{k+1}(t))^{-\frac{1}{2}}
        \end{pmatrix}}
    \end{equation*}
    down to $P_{m,2m,\dots ,km}(\mathbb{C}^n)$ is a Brownian motion on $P_{m,2m,\dots ,km}(\mathbb{C}^n)$, where $(\beta (t))_{t\geq 0}$ is a Brownian motion on $\mathbb{R}^{k+1}$ independent of $(w(t))_{t\geq 0}$.
\end{corollary}
\begin{proof}
    By the commutative diagram \eqref{eq:cd-canonical-torus-bundle}, the projection of the process $(X(t))_{t\geq 0}$, given in theorem \ref{thm:horizontal-MB-m}, down to $P_{m,2m,\dots ,km}(\mathbb{C}^n)$ is a horizontal Brownian motion on the fibration corresponding to the canonical torus bundle.
    Note that the process
    \begin{equation*}
        Y(t) :=\scalemath{0.95}{\begin{pmatrix}
            e^{i\frac{\mathfrak{a}_1(t)}{m}} w_{11}(t)(I_m +w_1^*(t)w_1(t))^{-\frac{1}{2}} & \dots & e^{i\frac{\mathfrak{a}_{k+1}(t)}{m}} w_{1(k+1)}(t)(I_m +w_{k+1}^*(t)w_{k+1}(t))^{-\frac{1}{2}} \\
            \vdots & \ddots & \vdots \\
            e^{i\frac{\mathfrak{a}_1(t)}{m}}w_{k1}(t)(I_m +w_1^*(t)w_1(t))^{-\frac{1}{2}} & \dots & e^{i\frac{\mathfrak{a}_{k+1}(t)}{m}}w_{k(k+1)}(t)(I_m +w_{k+1}^*(t)w_{k+1}(t))^{-\frac{1}{2}} \\
            e^{i\frac{\mathfrak{a}_1(t)}{m}}(I_m +w_1^*(t)w_1(t))^{-\frac{1}{2}} & \dots & e^{i\frac{\mathfrak{a}_{k+1}(t)}{m}}(I_m +w_{k+1}^*(t)w_{k+1}(t))^{-\frac{1}{2}}
        \end{pmatrix}}
    \end{equation*}
    projects down to the same process as $(X(t))_{t\geq 0}$.
    
    The Laplace--Beltrami operator on $P_{m,2m,\dots ,km}(\mathbb{C}^n)$ satisfies
    \begin{equation}\label{eq:decomposition-Laplace-canonical-torus}
        \Delta_{P_{m,2m,\dots ,km}(\mathbb{C}^n)} =\Delta_{\mathcal{H}} +\sum_{j=1}^{k+1}\frac{\partial^2}{\partial\theta_j^2},
    \end{equation}
    where $\Delta_{\mathcal{H}}$ is the horizontal Laplacian of the fibration \eqref{eq:fibration-canonical-torus-bundle} which commutes with $\sum_{j=1}^{k+1}\frac{\partial^2}{\partial\theta_j^2}$ since the corresponding fibration is totally geodesic \cite[Theorem 3.1.18.]{Baudoin2024-is}.
    The conclusion follows.
\end{proof}

\begin{corollary}\label{cor:connection-process-BM}
    Let $\eta$ be the connection form of the canonical torus bundle given in lemma \ref{lemma:connection-form-canonical-torus-bundle} and $(\tilde{w}(t))_{t\geq 0}$ a Brownian motion on $P_{m,2m,\dots ,km}(\mathbb{C}^n)$.
    The process
    \begin{equation*}
        \eta (t):=\frac{1}{m}\int_{\tilde{w}[0,t]}\eta
    \end{equation*}
    is a Brownian motion on $\mathbb{R}^{k+1}$.
\end{corollary}
\begin{proof}
    By corollary \ref{cor:Brownian-motion-torus-bundle} we have
    \begin{equation}\label{eq:fundamental-relation-processes}
        \theta (t)=\mathfrak{a}(t) +m\beta (t),
    \end{equation}
    where $\theta (t):=\int_{\tilde{w}[0,t]} d\theta$.
    The conclusion now follows from the relation \eqref{eq:fundamental-decomposition-m}.
\end{proof}

\begin{remark}\label{rmk:winding-number-angular-part}
Let $(U(t))_{t\geq 0}$ be a Brownian motion on $\mathbf{U}(n)$.
We have in distribution
\begin{equation*}
    \det (U_{j(k+1)}(t)) =e^{i\theta_j (t)}\det (w^*_{j(k+1)}(t)w_{j(k+1)}(t) +I_m)^{-\frac{1}{2}} ,
\end{equation*}
where $(w(t))_{t\geq 0}$ is the projection of $(U(t))_{t\geq 0}$ onto $F_{m,2m,\dots ,km}(\mathbb{C}^n)$, which is a Brownian motion on $F_{m,2m,\dots ,km}(\mathbb{C}^n)$, $\theta_j (t):=\int_{U[0,t]}d\theta_j$ is the \emph{$j^{\textrm{th}}$-winding process} and $d\theta$ the winding form from definition \ref{def:winding-form-blocks-of-equal-size}.
This means that the $j^{\textrm{th}}$-winding process around $\mathcal{D}_m^c$ is the angular part of the complex number $\det (U_{j(k+1)}(t))$.

This follows from corollary \ref{cor:Brownian-motion-torus-bundle} after noting that a horizontal Brownian motion on $\mathbf{U}(n)$ with respect to the fibration $\mathbf{SU}(m)^{k+1}\rightarrow\mathbf{U}(n)\rightarrow P_{m,2m,\dots ,km}(\mathbb{C}^n)$ is given by multiplying $(\tilde{X}(t))_{t\geq 0}$ with a $\mathbf{SU}(m)^{k+1}$-valued process from the right by \cite[Theorem 3.1.10]{Baudoin2024-is}.
This then gives a Brownian motion on $\mathbf{U}(n)$ after multiplying by an independent $\mathbf{SU}(m)^{k+1}$-valued Brownian motion as can be seen from a decomposition of $\Delta_{\mathbf{U}(n)}$ similar to the one in \eqref{eq:decomposition-Laplace-canonical-torus}.
\end{remark}

\subsection{Characteristic function of the stochastic area functional}
In this section we determine the stochastic area processes on the partial flag manifold with blocks of equal size $F_{m,2m,\dots ,km}(\mathbb{C}^n)$.
We begin with a computational lemma.

\begin{lemma}\label{lemma:derivatives-matrices}
    On the domain $\{A\in\mathbb{C}^{n\times n}\mid\det (A)\neq 0\}$ we have
    \begin{equation*}
        \frac{\partial (A^{-1})_{\alpha\beta}}{\partial A_{pq}} =-A_{\alpha p}^{-1}A_{q\beta}^{-1}
    \end{equation*}
    for $1\leq\alpha ,\beta ,p,q\leq n$.
    Furthermore
    \begin{equation*}
        \frac{\partial\det (A)}{\partial A_{pq}} =\det (A) A^{-1}_{qp}
    \end{equation*}
    for $1\leq p,q\leq n$.
\end{lemma}
\begin{proof}
    For the first equation we take the partial derivative of the relation
    \begin{equation*}
        \sum_{u=1}^n A^{-1}_{\alpha u}A_{uv}=\delta_{\alpha v}
    \end{equation*}
    with respect to $A_{pq}$ to obtain
    \begin{equation*}
        \sum_{u=1}^n\frac{\partial A^{-1}_{\alpha u}}{\partial A_{pq}} A_{uv } =-A^{-1}_{\alpha p}\delta_{qv} .
    \end{equation*}
    Multiplying this equation by $A^{-1}_{v\beta }$ and summing over $1\leq v\leq n$ gives the first equation in the lemma.

    Note that by Cramer's rule the cofactor matrix $C :=\det (A)(A^{-1})^T$ of $A$ is equal to the minor matrix of $A$ up to sign.
    In particular, $C_{pu}$ is independent of $A_{pq}$.
    For the second equation we expand the determinant along the $p^{\textrm{th}}$-row, i.e.
    \begin{equation*}
        \det (A)=\sum_{u=1}^n A_{p u} C_{pu}.
    \end{equation*}
    Taking the derivative of the above expression with respect to $A_{pq}$ immediately gives the second equation in the lemma.
\end{proof}

\begin{theorem}\label{thm:Laplace-transform-area-functional-equal-size}
    Let $(\Lambda (t),\mathfrak{a} (t))_{t\geq 0}$ be the diffusion process from before.
    For any
    \begin{equation*}
        u=(u_1,\dots ,u_{k+1})\in\mathbb{R}^{k+1},
    \end{equation*}
    any matrices in the matrix simplex $\Lambda (0),\Lambda\in\mathcal{T}^m_{k+1}$ and any $t>0$, we have
    \begin{align*}
        \mathbb{E}\left( e^{i\sum_{j=1}^{k+1} u_j\mathfrak{a}_j(t)}\bigm| \Lambda (t)=\Lambda\right) =&e^{-m(n-m)\sum_{j=1}^{k+1}|u_j|t-\frac{m}{2}\sum_{1\leq j\neq\ell\leq k+1}(u_ju_{\ell} +|u_ju_{\ell} |)t}\\
        &\prod_{j=1}^{k+1}\left(\frac{\det (\Lambda_j(0))}{\det (\Lambda_j)}\right)^{\frac{|u_j|}{2}}\frac{\hat{q}_{2t}^{(m/2 +|u_1|,\dots ,m/2+|u_{k+1}|)} (\Lambda (0) ,d\Lambda )}{\hat{q}_{2t}^{(m/2,\dots ,m/2)} (\Lambda (0) ,d\Lambda )} ,
    \end{align*}
    where $\hat{q}_t^{(\kappa_1,\dots ,\kappa_{k+1})}$ is the density of a Jacobi operator on the simplex of Hermitian matrices of index $(\kappa_1 ,\dots ,\kappa_{k+1} )$, i.e. a diffusion process with generator
    \begin{align*}
        \mathcal{G}_{\kappa_1,\dots ,\kappa_{k+1}}^m :=& 2\sum_{j=1}^{k+1}\sum_{\alpha ,\beta ,\gamma ,\delta =1}^m (I_m-\Lambda_j)_{\alpha\delta}\Lambda_{j,\gamma\beta}\frac{\partial^2}{\partial\Lambda_{j,\gamma\delta}\partial\Lambda_{j,\alpha\beta}} \\ 
        &-2\sum_{1\leq j\neq\ell\leq k+1}\sum_{\alpha ,\beta ,\gamma ,\delta =1}^{m}(\Lambda_{j,\gamma\beta}\Lambda_{\ell ,\alpha\delta} +\Lambda_{j,\alpha\delta}\Lambda_{\ell ,\gamma\beta})\frac{\partial^2}{\partial\Lambda_{j,\alpha\beta}\partial\Lambda_{\ell ,\gamma\delta}} \\
        &+2\sum_{j=1}^{k+1}\sum_{\alpha ,\beta =1}^{m}\left(\left(\kappa_j +\frac{m}{2}\right) \delta_{\alpha\beta}-\left(|\kappa |+\frac{n}{2}\right)\Lambda_{j,\alpha\beta}\right)\frac{\partial}{\partial\Lambda_{j,\alpha\beta}} ,
    \end{align*}
    with respect to the Lebesgue measure on the entries of $\Lambda$.
\end{theorem}
\begin{proof}
    From corollary \ref{cor:main-diffusion-m}, we know that conditioned on $(\Lambda (t))_{0\leq s\leq t}$, the stochastic area process $(\mathfrak{a} (t))_{t\geq 0}$ is a Gaussian random variable with mean zero and covariance matrix
    \begin{equation*}
        \Sigma (t) :=\begin{pmatrix}
            \int_0^t(\mathrm{Tr}(\Lambda_1^{-1}(s)) -m)ds & \dots & mt\\
            \vdots & \ddots & \vdots \\
            mt & \dots & \int_0^t(\mathrm{Tr}((\Lambda_{k+1}^{-1}(s)) -m)ds
        \end{pmatrix} .
    \end{equation*}
    Therefore
    \begin{align*}
        \mathbb{E}\bigg( &e^{i\sum_{j=1}^{k+1} u_j\mathfrak{a}_j(t)} \mid \Lambda (t)=\Lambda\bigg)
        =\mathbb{E}\left( e^{-\frac{1}{2}u^T\Sigma (t) u}\mid\Lambda (t) =\Lambda\right) \\
        &\quad =\mathbb{E}\left(\exp\left( -\sum_{j=1}^{k+1} u_j^2\int_0^t(\mathrm{Tr}(\Lambda_j^{-1}(s)) -m)ds -\frac{m}{2}\sum_{1\leq j\neq\ell\leq k+1}u_ju_{\ell}t\right)\biggm|\Lambda (t) =\Lambda\right) .
    \end{align*}
    Define the function $f:\mathcal{T}_{k+1}^m\rightarrow\mathbb{R}$ by
    \begin{equation*}
        f(\Lambda_1,\dots ,\Lambda_{k+1}) :=\prod_{j=1}^{k+1}(\det (\Lambda_j ))^{\frac{|u_j|}{2}} 
    \end{equation*}
    and let $L$ be the generator of $(\Lambda (t))_{t\geq 0}$.
    Using lemma \ref{lemma:derivatives-matrices} we calculate
    \begin{align*}
        Lf =&2\sum_{j=1}^{k+1}\sum_{\alpha ,\beta ,\gamma ,\delta =1}^m (I_m-\Lambda_j)_{\alpha\delta}\Lambda_{\gamma\beta}\frac{\partial^2 f}{\partial\Lambda_{j,\gamma\delta}\partial\Lambda_{j,\alpha\beta}} +2\sum_{j=1}^{k+1}\sum_{\alpha ,\beta =1}^m(mI_m -n\Lambda_j)_{\alpha\beta}\frac{\partial f}{\partial\Lambda_{j,\alpha\beta}} \\
        &-2\sum_{1\leq j\neq\ell\leq k+1}\sum_{\alpha ,\beta,\gamma ,\delta =1}^m(\Lambda_{j,\gamma\beta}\Lambda_{\ell ,\alpha\delta} +\Lambda_{j,\alpha\delta}\Lambda_{\ell ,\gamma\beta})\frac{\partial^2 f}{\partial\Lambda_{j,\alpha\beta}\partial\Lambda_{\ell ,\gamma\delta}} \\
        =&2\sum_{j=1}^{k+1}\sum_{\alpha ,\beta ,\gamma ,\delta =1}^m(I_m-\Lambda_j)_{\alpha\delta}\Lambda_{j,\gamma\beta}\left(\frac{|u_j|}{2}(\Lambda^{-1}_{j})_{\beta\gamma}(\Lambda_j^{-1})_{\delta\alpha} +\frac{|u_j|^2}{4}(\Lambda_j^{-1})_{\beta\alpha}(\Lambda_j^{-1})_{\delta\gamma}\right) f\\
        &-|u_j|\sum_{j=1}^{k+1}\sum_{\alpha ,\beta =1}^m(mI_m -n\Lambda_j)_{\alpha\beta}(\Lambda_j^{-1})_{\beta\alpha} f\\
        &-\sum_{1\leq j\neq\ell\leq k+1}\sum_{\alpha ,\beta ,\gamma ,\delta =1}^m\frac{|u_j u_{\ell }|}{4}(\Lambda_{j,\gamma\beta}\Lambda_{\ell ,\alpha\delta} +\Lambda_{j,\alpha\delta}\Lambda_{\ell ,\gamma\beta})(\Lambda_j^{-1})_{\beta\alpha}(\Lambda_{\ell}^{-1})_{\delta\gamma} f\\
        =&m\sum_{j=1}^{k+1}|u_j|(\mathrm{Tr} (\Lambda_j^{-1}) -m)f +\sum_{j=1}^{k+1}\frac{|u_j|^2}{2}(\mathrm{Tr}(\Lambda_j^{-1})-m)f\\
        &-\sum_{j=1}^{k+1}|u_j|(m\mathrm{Tr}(\Lambda_j^{-1}) -nm)f 
        -\frac{m}{2}\sum_{1\leq j\neq\ell\leq k+1}|u_{j}u_{\ell}| f\\
        =&-m(n-m)|u|f +\sum_{j=1}^{k+1}\frac{|u_j|^2}{2}(\mathrm{Tr}(\Lambda_j^{-1}) -m)f -\frac{m}{2}\sum_{1\leq j\neq\ell\leq k+1 }|u_{j}u_{\ell }|f .
    \end{align*}
    From which we see that $f$ is an eigenfunction of $L-\sum_{j=1}^{k+1}\frac{|u_j|^2}{2} (\mathrm{Tr}(\Lambda_j^{-1}) -m)$ with eigenvalue
    \begin{equation*}
        -m(n-m)|u|-\frac{m}{2}\sum_{1\leq j\neq\ell\leq k+1}|u_{\ell} u_j| .
    \end{equation*}
    Itô's formula now shows that the process
    \begin{equation*}
        D_t^u :=e^{m(n-m)\sum_{j=1}^{k+1}|u_j|t +\frac{m}{2}\sum_{1\leq j\neq\ell\leq k+1}|u_{\ell} u_j|t}\left(\prod_{j=1}^{k+1}\left(\frac{\det (\Lambda_j)}{\det (\Lambda_j(0))}\right)^{\frac{|u_j|}{2}} e^{-\frac{u_j^2}{2}\int_0^t (\mathrm{Tr} (\Lambda_j^{-1})-m )ds}\right)
    \end{equation*}
    is a local martingale.
    The bound
    \begin{equation*}
        D_t^u\leq e^{m(n-m)\sum_{j=1}^{k+1}|u_j|t +\frac{m}{2}\sum_{j\neq\ell}|u_{\ell} u_j|t+\frac{1}{2}|u_j|^2mt}\prod_{j=1}^{k+1}\det (\Lambda_j(0))^{-|u_j|/2}
    \end{equation*}
    shows that it is in fact a martingale.
    Define a new probability measure $\mathbb{P}^u$ by $d\mathbb{P}^{u} :=D_t^ud\mathbb{P}$.
    We have for every bounded Borel function $F$ that
    \begin{equation}
    \begin{split}\label{eq:intermidiate-formula-characteristic-function}
        &\mathbb{E}\left(F(\Lambda_1(t),\dots ,\Lambda_{k+1}(t))e^{-\frac{1}{2}\sum_{j=1}^{k+1} u_j^2\int_0^t(\mathrm{Tr}(\Lambda_j(s))^{-1})-m) ds}\right) \\
        &=e^{-m(n-m)\sum_{j=1}^{k+1}|u_j|t -\frac{m}{2}\sum_{1\leq j\neq\ell\leq k+1} |u_ju_{\ell}|t}\mathbb{E}^u\left(\scalemath{0.9}{\frac{F(\Lambda_1(t) ,\dots ,\Lambda_{k+1}(t))}{\prod_{j=1}^{k+1}(\det (\Lambda_j(t)))^{\frac{|u_j|}{2}}}}\right)\prod_{j=1}^{k+1}(\det (\Lambda_j(0)))^{\frac{|u_j|}{2}} .
    \end{split}
    \end{equation}
    Let $s_t^{(u)}(\Lambda (0),d\Lambda )$ denote the probability distribution of $\Lambda (t)$ under $\mathbb{P}^u$ (where we use the convention $\mathbb{P}^0 =\mathbb{P}$).
    The above equality then implies
    \begin{align*}
        \mathbb{E}\bigg( &e^{-\frac{1}{2}\sum_{j=1}^{k+1}u_j^2\int_0^t(\mathrm{Tr}(\Lambda_j^{-1}(s)) -m)ds}\Bigm| \Lambda (t) =\Lambda\bigg) s_t^{(0)}(\Lambda (0) ,d\Lambda )\\
        &\qquad =e^{-m(n-m)\sum_{j=1}^{k+1}|u_j|t -\frac{m}{2}\sum_{1\leq j\neq\ell\leq k+1} |u_ju_{\ell} |t}\prod_{j=1}^{k+1}\left(\frac{\det (\Lambda_j(0))}{\det (\Lambda_j)}\right)^{\frac{|u_j|}{2}} s_t^{(u)}(\Lambda (0),d\Lambda ).
    \end{align*}
    We will now calculate the stochastic differential equation satisfied by $(\Lambda (t))_{t\geq 0}$ under $\mathbb{P}^u$.
    By Itô's formula, theorem \ref{thm:generator-process-of-interest} and lemma \ref{lemma:derivatives-matrices} we have
    \begin{align*}
        d\ln (\det (\Lambda_j)) =&\frac{d\det (\Lambda_j)}{\det (\Lambda_j)} -\frac{1}{2}\frac{d\det (\Lambda_j) d\det (\Lambda_j)}{\det (\Lambda_j)^2} 
        =\mathrm{Tr} (\Lambda_j^{-1}d\Lambda_j) -\frac{1}{2}\mathrm{Tr}(\Lambda_j^{-1}d\Lambda_j)^2\\
        &+\frac{1}{2}\sum_{p,q,p',q'=1}^m\left( \Lambda_{j,qp}^{-1}\Lambda_{j,q'p'}^{-1} -\Lambda_{j,q'p}^{-1}\Lambda_{j,qp'}^{-1}\right) d\Lambda_{j,pq}d\Lambda_{j,p'q'}\\
        =&\sum_{\ell =1;\ell\neq j}^{k+1}\mathrm{Tr}(\Lambda_j^{-\frac{1}{2}}d\gamma^*_{\ell j}\Lambda_{\ell}^{\frac{1}{2}})
        +\sum_{\ell =1;\ell\neq j}^{k+1}\mathrm{Tr}(\Lambda_{\ell}^{\frac{1}{2}}d\gamma_{\ell j}\Lambda_{j}^{-\frac{1}{2}})
        +2m(\mathrm{Tr}(\Lambda_j^{-1} )-n) dt\\
        &+2\sum_{\substack{\ell ,u=1\\ \ell ,u\neq j}}^{k+1}\left( \delta_{\ell j}\delta_{ju}\mathrm{Tr}\left(\Lambda_j^{-\frac{1}{2}}\Lambda_{u}^{\frac{1}{2}}\right)\mathrm{Tr}\left(\Lambda_{\ell}^{\frac{1}{2}}\Lambda_j^{-\frac{1}{2}}\right) 
        -\delta_{u\ell }\mathrm{Tr}\left( \Lambda_j^{-1}\right)\mathrm{Tr}\left(\Lambda_{\ell}^{\frac{1}{2}}\Lambda_u^{\frac{1}{2}}\right)\right) dt\\
        =&\sum_{\ell =1;\ell\neq j}^{k+1}\mathrm{Tr}(\Lambda_j^{\frac{1}{2}}(d\gamma_{\ell j})^*\Lambda_{\ell}^{-\frac{1}{2}}) +
        \sum_{\ell =1;\ell\neq j}^{k+1}\mathrm{Tr}(\Lambda_{\ell}^{-\frac{1}{2}}d\gamma_{\ell j}\Lambda_{j}^{\frac{1}{2}})
        -2m(n-m) dt
    \end{align*}
    for some Brownian motion $(\gamma(t))_{t\geq 0}$ on $\mathfrak{u}(n)$.
    This gives
    \begin{align*}
        \prod_{j=1}^{k+1}&(\det (\Lambda_j))^{\frac{|u_j|}{2}} 
        =\exp\left(\sum_{1\leq\ell\neq j\leq k+1}\frac{|u_j|}{2}\int_0^t\mathrm{Tr}(\Lambda_j^{-\frac{1}{2}}(d\gamma_{\ell j})^*\Lambda_{\ell}^{\frac{1}{2}})\right) \\
        &\qquad\exp\left(\sum_{1\leq\ell\neq j\leq k+1}\frac{|u_j|}{2}\int_0^t\mathrm{Tr}(\Lambda_{\ell}^{\frac{1}{2}}d\gamma_{\ell j}\Lambda_{j}^{-\frac{1}{2}})\right)
        \exp\left( -m(n-m)\sum_{j=1}^{k+1}|u_j|\right) .
    \end{align*}
    We can write the above using the inner product $\langle A,B\rangle :=\mathrm{Re}(\mathrm{Tr} (AB^*))$ on the space of matrices.
    Define
    \begin{equation*}
        \Theta_{\ell j}(t) :=|u_j|\Lambda_{\ell}^{\frac{1}{2}}(t)\Lambda_j^{-\frac{1}{2}}(t) -|u_{\ell}|\Lambda_{\ell}^{-\frac{1}{2}}(t)\Lambda_j^{\frac{1}{2}}(t)
    \end{equation*}
    for $1\leq\ell <j\leq k+1$.
    By Girsanov's theorem the stochastic processes
    \begin{equation*}
        d\tilde{\gamma}_{\ell j}(t) :=d\gamma_{\ell j}(t) -\Theta_{\ell j}(t) dt
    \end{equation*}
    for $1\leq\ell <j\leq k+1$, are independent Brownian motions on $\mathbb{C}^{m\times m}$.
    Define furthermore $d\gamma_{j\ell}(t) :=-d\gamma_{\ell j}^*(t)$ for $1\leq\ell <j\leq k+1$.
    Under $\mathbb{P}^u$, the process $(\Lambda (t))_{t\geq 0}$ satisfies the stochastic differential equation
    \begin{align*}
        d\Lambda_j =& \Lambda_j^{\frac{1}{2}}\sum_{1\leq\ell\neq j\leq k+1}d\gamma^*_{\ell j}\Lambda_{\ell}^{\frac{1}{2}}
        +\sum_{1\leq\ell\neq j\leq k+1}\Lambda_{\ell}^{\frac{1}{2}}d\gamma_{\ell j}\Lambda_j^{\frac{1}{2}} -2n\Lambda_jdt +2mI_mdt\\
        &+2\sum_{1\leq\ell <j}(|u_j|\Lambda_{\ell} -|u_{\ell}|\Lambda_j) dt
        -2\sum_{j<\ell\leq k+1} (|u_{\ell}|\Lambda_j -|u_j|\Lambda_{\ell}) dt\\
        =&\Lambda_j^{\frac{1}{2}}\sum_{1\leq\ell\neq j\leq k+1}d\gamma^*_{\ell j}\Lambda_{\ell}^{\frac{1}{2}}
        +\sum_{1\leq\ell\neq j\leq k+1}\Lambda_{\ell}^{\frac{1}{2}}d\gamma_{\ell j}\Lambda_j^{\frac{1}{2}} +2(mI_m-n\Lambda_j -|u|\Lambda_j +|u_j|I_m)dt ,
    \end{align*}
    where we used that $\sum_{\ell =1}^{k+1}\Lambda_j(t) =I_m$ and the notation $|u|:=\sum_{i=1}^{k+1}|u_{i}|$.
    In particular, the generator of $(\Lambda (t))_{t\geq 0}$ under $\mathbb{P}^u$ is given by twice a Jacobi operator on the simplex of Hermitian matrices:
    \begin{multline*}
        2\sum_{j=1}^{k+1}\sum_{\alpha ,\beta ,\gamma ,\delta =1}^m (I_m-\Lambda_j)_{\alpha\delta}\Lambda_{j,\gamma\beta}\frac{\partial^2}{\partial\Lambda_{j,\gamma\delta}\partial\Lambda_{j,\alpha\beta}}
        +2\sum_{j=1}^{k+1}\sum_{\alpha ,\beta =1}^m ((m+|u_j|)I_m-(n+|u|)\Lambda_j)_{\alpha\beta}\frac{\partial}{\partial\Lambda_{j,\alpha\beta}}\\
        -\sum_{1\leq j\neq\ell\leq k+1}\sum_{\alpha ,\beta ,\gamma ,\delta =1}^m(\Lambda_{j,\gamma\beta}\Lambda_{\ell ,\alpha\delta} +\Lambda_{j,\alpha\delta}\Lambda_{\ell ,\gamma\beta})\frac{\partial^2}{\partial\Lambda_{j,\alpha\beta}\partial\Lambda_{\ell ,\gamma\delta}}.
    \end{multline*}
\end{proof}

\subsection{Limit theorem}
The previous theorem allows us to determine the large time asymptotics of the stochastic area processes.

\begin{corollary}\label{cor:limit-area-process-on-Fm}
    Let $(\mathfrak{a}(t))_{t\geq 0}$ be the stochastic area process as given in equation \eqref{eq:stochastic-area-process}.
    The following convergence holds in distribution
    \begin{equation*}
        \frac{\mathfrak{a}(t)}{t}\rightarrow\left( C_{m(n-m)}^1,\dots ,C_{m(n-m)}^{k+1}\right)\textrm{ as } t\rightarrow\infty ,
    \end{equation*}
    where $C^1_{m(n-m)},\dots ,C^{k+1}_{m(n-m)}$ are independent Cauchy random variables with parameter $m(n-m)$.
\end{corollary}
\begin{proof}
    Let $u_1,\dots ,u_{k+1}\in\mathbb{R}$.
    From theorem \ref{thm:Laplace-transform-area-functional-equal-size} we see that
    \begin{align*}
        \mathbb{E}\bigg( e^{i\sum_{j=1}^{k+1} u_j\frac{\mathfrak{a}_j (t)}{t}}\bigg)
        =&e^{-m(n-m)\sum_{j=1}^{k+1}|u_j| -\frac{m}{2t}\sum_{1\leq j\neq\ell\leq k+1}(u_ju_{\ell} +|u_ju_{\ell }|)}\\
        &\int_{\mathcal{T}_{k+1}^m}\prod_{j=1}^{k+1}\left(\frac{\det (\Lambda_j (0))}{\det (\Lambda_j)}\right)^{\frac{|u_j|}{2t}}\hat{q}_{2t}^{(m/2+|u_1|/t,\dots ,m/2+|u_{k+1}|/t)}(\Lambda (0),\Lambda) d\Lambda
    \end{align*}
    for $t>0$.
    
    We will now show that
    \begin{equation*}
        \int_{\mathcal{T}_{k+1}^m} \frac{\hat{q}_{2t}^{(m/2+|u_1|/t,\dots ,m/2+|u_{k+1}|/t)}(\Lambda (0),\Lambda)}{\prod_{j=1}^{k+1}\det (\Lambda_j)^{|u_j|/(2t)}} d\Lambda\rightarrow 1\textrm{ as } t\rightarrow\infty ,
    \end{equation*}
    which will directly show that the above characteristic function converges to the characteristic function of independent Cauchy distributions of parameter $m(n-m)$.
    Let $\varepsilon >0$ and define
    \begin{equation*}
        \mathcal{T}_{k+1}^{m,\varepsilon} :=\left\{ (\Lambda_1,\dots ,\Lambda_{k+1})\in M_{m\times m}(\mathbb{C})^{k+1}\Bigm| \sum_{j=1}^{k+1}\Lambda_j =I_m, \Lambda_j\geq\varepsilon I_m ,\Lambda_j^*=\Lambda_j\textrm{ for } 1\leq j\leq k+1\right\} ,
    \end{equation*}
    then
    \begin{align*}
        \Bigg|\int_{\mathcal{T}_{k+1}^m} &\frac{\hat{q}_{2t}^{(m/2+|u_1|/t,\dots ,m/2+|u_{k+1}|/t)}(\Lambda (0),\Lambda)}{\prod_{j=1}^{k+1}\det (\Lambda_j)^{|u_j|/(2t)}} d\Lambda -1\Bigg| \\
        &\qquad\qquad\leq \int_{\mathcal{T}_{k+1}^{m,\varepsilon}} \left|\prod_{j=1}^{k+1}\det (\Lambda_j)^{-\frac{|u_j|}{2t}} -1\right| \hat{q}_{2t}^{(m/2+|u_1|/t,\dots ,m/2+|u_{k+1}|/t)} (\Lambda (0),\Lambda)d\Lambda \\
        &\qquad\qquad\quad+\int_{(\mathcal{T}_{k+1}^{m,\varepsilon})^c}  \frac{\hat{q}_{2t}^{(m/2+|u_1|/t,\dots ,m/2+|u_{k+1}|/t)}(\Lambda (0),\Lambda)}{\prod_{j=1}^{k+1}\det (\Lambda_j)^{|u_j|/(2t)}} d\Lambda .
    \end{align*}
    Note that
    \begin{equation*}
        \int_{\mathcal{T}_{k+1}^{m,\varepsilon}} \left|\prod_{j=1}^{k+1}\det (\Lambda_j)^{-\frac{|u_j|}{2t}} -1\right| \hat{q}_{2t}^{(m/2+|u_1|/t,\dots ,m/2+|u_{k+1}|/t)} (\Lambda (0),\Lambda)d\Lambda\rightarrow 0\textrm{ as }t\rightarrow\infty ,
    \end{equation*}
    since $\prod_{j=1}^{k+1}\det (\Lambda_j)^{-\frac{|u_j|}{t}}\rightarrow 1$ uniformly on $\mathcal{T}_{k+1}^{m,\varepsilon}$ and because
    \begin{equation*}
        \int_{\mathcal{T}_{k+1}^{m,\varepsilon}} \hat{q}_{2t}^{(m/2+|u_1|/t,\dots ,m/2+|u_{k+1}|/t)} (\Lambda (0),\Lambda)d\Lambda\leq 1.
    \end{equation*}

    Furthermore, from equation \eqref{eq:intermidiate-formula-characteristic-function} with the function $F:=\mathds{1}_{(\mathcal{T}_{k+1}^{m,\varepsilon})^c}$ we have that
    \begin{align*}
        \int_{(\mathcal{T}_{k+1}^{m,\varepsilon})^c} &\frac{\hat{q}_{2t}^{(m/2+|u_1|/t,\dots ,m/2+|u_{k+1}|/t)}(\Lambda (0),\Lambda)}{\prod_{j=1}^{k+1}\det (\Lambda_j (0))^{\frac{|u_j|}{2t}}} d\Lambda
        =\mathbb{E}^u\left( \frac{\mathds{1}_{(\mathcal{T}_{k+1}^{m,\varepsilon})^c} (\Lambda_1(t) ,\dots ,\Lambda_{k+1}(t))}{\prod_{j=1}^{k+1}(\det (\Lambda_j(t)))^{\frac{|u_j|}{2t}}}\right) \\
        &=\frac{e^{m(n-m)\sum_{j=1}^{k+1}|u_j| +\frac{m}{2t}\sum_{1\leq j\neq\ell\leq k+1} |u_ju_{\ell}|}}{\prod_{j=1}^{k+1}(\det (\Lambda_j(0)))^{\frac{|u_j|}{2t}}}\mathbb{E} \left(\frac{\mathds{1}_{(\mathcal{T}_{k+1}^{m,\varepsilon})^c}(\Lambda_1(t),\dots ,\Lambda_{k+1}(t))}{e^{\frac{1}{2}\sum_{j=1}^{k+1} u_j^2\int_0^t(\mathrm{Tr}(\Lambda_j(s))^{-1})-m) ds}}\right) \\
        &\leq \frac{e^{m(n-m)\sum_{j=1}^{k+1}|u_j| +\frac{m}{2t}\sum_{1\leq j\neq\ell\leq k+1} |u_ju_{\ell}|}}{\prod_{j=1}^{k+1}(\det (\Lambda_j(0)))^{\frac{|u_j|}{2t}}}\mathbb{P} ((\Lambda_1(t),\dots ,\Lambda_{k+1}(t))\in (\mathcal{T}_{k+1}^{m,\varepsilon})^c) \\
        &\leq\frac{e^{m(n-m)\sum_{j=1}^{k+1}|u_j| +\frac{m}{2t}\sum_{1\leq j\neq\ell\leq k+1} |u_ju_{\ell}|}}{\prod_{j=1}^{k+1}(\det (\Lambda_j(0)))^{\frac{|u_j|}{2t}}}\sum_{v=1}^{k+1}\mathbb{P} (\Lambda_v(t)<\varepsilon I_m)
    \end{align*}
    for all $t>0$.
    By remark \ref{rmk:marginal-ditribution} $(\Lambda_j (t))_{t\geq 0}$ is a Hermitian Jacobi process of index $(\kappa_j -m/2 ,|\kappa|-\kappa_j-m/2)$.
    It is well known that these processes converge weakly to a complex matrix-variate distribution of index $(\kappa_j +m/2,|\kappa| -\kappa_j +m/2)$ as $t\rightarrow\infty$ \cite{demni2020hermitian}.
    The \emph{(type I) complex matrix-variate Beta distribution} of index $(\alpha ,\beta )$ is given by
    \begin{equation*}
        C_{\alpha ,\beta}|\det (\Lambda )|^{\alpha -m} |\det (I_m-\Lambda )|^{\beta -m}
    \end{equation*}
    for some normalisation constant $C_{\alpha ,\beta}>0$, see for example \cite[\textsection 5.3.1]{Mathai2022-xz}.
    
    Therefore,
    \begin{equation*}
        \limsup_{t\rightarrow\infty} \Biggl|\int_{\mathcal{T}_{k+1}^m}  \frac{\hat{q}_{2t}^{(m/2+|u_1|/t,\dots ,m/2+|u_{k+1}|/t)}(\Lambda (0),\Lambda)}{\prod_{j=1}^{k+1}\det (\Lambda_j)^{|u_j|/(2t)}} d\Lambda -1\Biggr|
        \leq e^{m(n-m)\sum_{j=1}^{k+1}|u_j|}\sum_{\ell =1}^{k+1}\mathbb{P} (B_{\ell }<\varepsilon I_m) ,
    \end{equation*}
    where $B_{\ell}$ is a complex matrix-variate Beta distribution of index $(\kappa_j +m/2,|\kappa| -\kappa_j +m/2)$ for $1\leq\ell\leq k+1$.
    Taking $\varepsilon\rightarrow 0$ yields the conclusion.
\end{proof}

\begin{remark}\label{rmk:limit-winding-process}
    By the relation \eqref{eq:fundamental-relation-processes} the winding processes $(\theta (t))_{t\geq 0}$, defined in the proof of corollary \ref{cor:connection-process-BM}, have the same asymptotics as the stochastic area processes, since almost surely $\beta (t)/t\rightarrow 0$ as $t\rightarrow\infty$.
\end{remark}

\section{Simultaneous windings on the Stiefel manifold}\label{chap:simulteneous-winding-Stiefel}
In this section, we present an application of our previous results to the study of simultaneous Brownian windings over the Grassmannian $Gr_{n,m}(\mathbb{C})$, i.e. the set of $m$-dimensional complex linear subspaces in $\mathbb{C}^n$, with $n=m(k+1)$ for some $k\geq 1$.
This application is a generalisation of both \cite[\textsection 5]{baudoin2025fullflag} and \cite[\textsection 8.3.3]{Baudoin2024-is}, as an added bonus we obtain an interpretation of the winding process defined in the latter.

\subsection{Winding over the Grassmannian}\label{sec:widing-over-Grassmannian}
As remarked in section \ref{sec:canonical-torus-bundle}, the discussion in that section works for general partial flag manifolds and thus in particular for the Grassmannians.
In this section we will discuss a straightforward modification for the Grassmannian $Gr_{n,m}(\mathbb{C})$, without going into all the details.
Note that the restriction $n=(k+1)m$ is not necessary for this section.
Recall that the Stiefel manifold is defined by
\begin{equation*}
    V_{n,m}(\mathbb{C}) =\{ M\in\mathbb{C}^{n\times m}\mid M^*M=I_m\} ,
\end{equation*}
and that we can make the identification
\begin{equation*}
    V_{n,m}(\mathbb{C})\cong\frac{\mathbf{U}(n)}{\mathbf{U}(n-m)} ,
\end{equation*}
where we identify $\mathbf{U}(n-m)$ with the subgroup
\begin{equation*}
    \left\{\begin{pmatrix} U & 0 \\ 0 & I_m\end{pmatrix}\in\mathbf{U}(n)\Bigm| U\in\mathbf{U}(n-m)\right\} .
\end{equation*}
The Stiefel manifold $V_{n,m}(\mathbb{C})$ can be interpreted as the first $m$-columns of a matrix in $\mathbf{U}(n)$.
The Stiefel fibration
\begin{equation*}
    \mathbf{U}(m)\rightarrow V_{n,m}(\mathbb{C})\rightarrow Gr_{n,m}(\mathbb{C})
\end{equation*}
is obtained from the action of $\mathbf{U}(m)$ on $V_{n,m}(\mathbb{C})$ given by right multiplication.

The \emph{canonical circle bundle} is defined as the quotient space
\begin{equation*}
    P_{n,m}(\mathbb{C}):=V_{n,m}(\mathbb{C})/\mathbf{SU}(m),
\end{equation*}
where the action is given by right multiplication.
We obtain the fibration
\begin{equation}\label{eq:Boothby-Wang-Grassmannian}
    \mathbf{U}(1)\rightarrow P_{n,m}(\mathbb{C})\rightarrow Gr_{n,m}(\mathbb{C}),
\end{equation}
which is a realisation of the Boothby--Wang fibration over $Gr_{n,m}(\mathbb{C})$.

The fibration \eqref{eq:Boothby-Wang-Grassmannian} can be locally trivialised by projecting the function
\begin{equation}
    \begin{split}\label{eq:trivialisation-torus-bundle-Gr}
        \mathbb{R}\times\mathbb{C}^{(n-m)\times m} &\mapsto\qquad\qquad\hat{V}_{n,m}(\mathbb{C})\\
        (\theta\, ,\quad M)\qquad &\rightarrow e^{i\frac{\theta}{m}}\begin{pmatrix} M\\ I_m\end{pmatrix}(M^*M+I_m)^{-\frac{1}{2}} ,
    \end{split}
\end{equation}
where
\begin{equation*}
    \hat{V}_{n,m}(\mathbb{C}) :=\left\{\begin{pmatrix} M \\ N \end{pmatrix}\in V_{n,m}(\mathbb{C})\Bigm| N\in\mathbb{C}^{m\times m},\det (N)\neq 0\right\} ,
\end{equation*}
down to an open dense subset of $P_{n,m}(\mathbb{C})$.
Since $\mathbb{C}^{(n-m)\times m}$ is simply connected, the fundamental group of $\hat{P}_{n,m}(\mathbb{C})$ is $\mathbb{Z}$.
We can therefore call the form $d\theta$ on $\hat{P}_{n,m}(\mathbb{C})$ the \emph{winding form} around the set $P_{n,m}(\mathbb{C})\backslash\pi_P(\hat{V}_{n,m}(\mathbb{C}))$.

Let $(\tilde{w}(t))_{t\geq 0}$ be a Brownian motion on $P_{n,m}(\mathbb{C})$.
Similarly to remark \eqref{rmk:winding-number-angular-part}, we see that in distribution
\begin{equation*}
    \det (Z(t)) =e^{i\theta (t)}\det ((w^*(t)w(t)+I_m)^{-\frac{1}{2}}) ,
\end{equation*}
where $(w(t))_{t\geq 0}$ is the projection of $(\tilde{w}(t))_{t\geq 0}$ down to $Gr_{n,m}(\mathbb{C})$, which is a Brownian motion on $Gr_{n,m}(\mathbb{C})$, $\theta (t):=\int_{\tilde{w}[0,t]}d\theta$ is the \emph{winding process}, and $(W(t),Z(t))_{t\geq 0}$ is a Brownian motion on $V_{n,m}(\mathbb{C})$ with $Z(t)\in\mathbb{C}^{m\times m}$.
This gives an interpretation of the winding number defined in \cite[\textsection 8.3.3]{Baudoin2024-is}.

\subsection{Simultaneous Brownian winding on the Stiefel manifold}
In this section we will again assume $n:=(k+1)m$.
More generally, the group $\mathbf{U}(1)^{k+1}$ acts isometrically on $V_{n,m}(\mathbb{C})$ by
\begin{equation*}
    (e^{i\theta_1},\dots ,e^{i\theta_{k+1}})(M_1,\dots ,M_{k+1})^T
    =(e^{i\theta_1} M_1,\dots ,e^{i\theta_{k+1}}M_{k+1})^T ,
\end{equation*}
where $M_j\in\mathbb{C}^{m\times m}$ for $1\leq j\leq k+1$.
This yields the following totally geodesic fibration
\begin{equation}\label{eq:simulteneous-winding-fibration}
    \mathbf{U}(1)^{k+1}\rightarrow V_{n,m}(\mathbb{C})\rightarrow\frac{\mathbf{U}(n)}{\mathbf{U}(n-m)\times\mathbf{U}(1)^{k+1}} .
\end{equation}

Let $\begin{pmatrix} Z_1& \dots &Z_{k+1}\end{pmatrix}^T$ be Brownian motion on the Stiefel manifold, with $Z_j\in\mathbb{C}^{m\times m}$ an invertible matrix for $1\leq j\leq k+1$.
Using other trivialisations of the circle bundle $P_{n,m}(\mathbb{C})$, which are given by \eqref{eq:trivialisation-torus-bundle-Gr} with the identity matrix in rows $(j-1)m+1$ through $jm$ for $1\leq j\leq k+1$ instead of the last $m$ rows and the range adjusted accordingly, we see that the angular parts $(\theta_1 (t),\dots ,\theta_{k+1}(t))_{t\geq 0}$ of
\begin{equation*}
    (\det (Z_1(t)),\dots ,\det (Z_{k+1}(t)))_{t\geq 0}
\end{equation*}
can be seen as a simultaneous winding numbers on $P_{n,m}(\mathbb{C})\backslash\{\det (Z_j)=0\textrm{ for } 1\leq j\leq k\}$.

To calculate the simultaneous winding number we use the commutative diagram
\begin{equation}\label{eq:diagram-simultaneous-winding}
\begin{tikzcd}[ampersand replacement=\&]
    {\mathbf{U}(n)} \&\& {F_{m,2m,\dots ,km}(\mathbb C^n)} \\
    \\
    V_{n,m}(\mathbb{C}) \&\& {\frac{\mathbf{U}(n)}{\mathbf{U}(n-m)\times      \mathbf{U}(1)^{k+1}}}
    \arrow[from=1-1, to=1-3]
    \arrow[from=1-1, to=3-1]
    \arrow[from=1-3, to=3-3]
    \arrow[from=3-1, to=3-3]
\end{tikzcd} ,
\end{equation}
where the action of $\mathbf{U}(n-m)\times\mathbf{SU}(m)^{k+1}$ on $F_{m,2m,\dots ,km}(\mathbb{C}^n)$ is given by
\begin{equation*}
    (V,U_1,\dots ,U_{k+1})\left[\begin{pmatrix} M_1 &\dots & M_{k+1}\\ Z_1 & \dots & Z_{k+1}\end{pmatrix}\right] =\left[\begin{pmatrix} VM_1 &\dots & VM_{k+1}\\ U_1Z_1 & \dots & U_{k+1}Z_{k+1}\end{pmatrix}\right]
\end{equation*}
and the map $\mathbf{U}(n)\rightarrow V_{n,m}(\mathbb{C})$ is the projection onto the last $m$ rows.

\begin{theorem}\label{thm:Brownian-motion-Stiefel-manifold}
    Let $(w (t))_{t\geq 0}$ be a Brownian motion on $F_{m,2m,\dots ,km}(\mathbb{C}^n)$ with stochastic area process $(\mathfrak{a}(t))_{t\geq 0}$.
    Let $(\beta (t))_{t\geq 0}$ a Brownian motion on $\mathbb{R}^{k+1}$ independent from $(w(t))_{t\geq 0}$.
    The $V_{n,m}(\mathbb{C})$-valued process
    \begin{equation*}
        X(t):=\begin{pmatrix} e^{i(\beta_1 (t)+\mathfrak{a}_1(t)/m)} (w^*_1(t)w_1(t) +I_m)^{-\frac{1}{2}}\Theta_1 (t) \\ \vdots \\ e^{i(\beta_{k+1}(t) +\mathfrak{a}_{k+1}(t)/m)}(w^*_{k+1}(t)w_{k+1}(t) +I_m)^{-\frac{1}{2}}\Theta_{k+1}(t)\end{pmatrix}
    \end{equation*}
    is a horizontal Brownian motion on $V_{n,m}(\mathbb{C})$ with respect to the fibration \eqref{eq:simulteneous-winding-fibration}, where $(\Theta (t))_{t\geq 0}$ is the $\mathbf{SU}(m)^{k+1}$-valued stochastic process from theorem \ref{thm:horizontal-MB-m}.
\end{theorem}
\begin{proof}
    Consider the commutative diagram \eqref{eq:diagram-simultaneous-winding} and recall that the map $\mathbf{U}(n)\rightarrow V_{n,m}(\mathbb{C})$ is given by projection onto the last $m$ columns.
    Together with theorem \ref{thm:horizontal-MB-m} this shows that a horizontal Brownian motion on the fibration
    \begin{equation}\label{eq:fibration-simultaneous-winding}
        \mathbf{U}(1)^{k+1}\rightarrow V_{n,m}(\mathbb{C})\rightarrow\frac{\mathbf{U}(n)}{\mathbf{U}(n-m)\times \mathbf{U}(1)^{k+1}}
    \end{equation}
    is given by
    \begin{equation*}
        \begin{pmatrix} e^{i\frac{\mathfrak{a}_1(t)}{m}}(w^*_1(t)w_1(t) +I_m)^{-\frac{1}{2}}\Theta_1(t) & \dots & e^{i\frac{\mathfrak{a}_{k+1}(t)}{m}}(w^*_{k+1}(t)w_{k+1}(t) +I_m)^{-\frac{1}{2}}\Theta_{k+1}(t)\end{pmatrix} .
    \end{equation*}
    The theorem now follows from the fact that
    \begin{equation*}
        \Delta_{V_{n,m}(\mathbb{C})} =\Delta_{\mathcal{H}'} +\sum_{j=1}^{k+1}\frac{\partial^2}{\partial\theta_j^2},
    \end{equation*}
    where $\Delta_{\mathcal{H}'}$ is the horizontal Laplacian of the fibration \eqref{eq:fibration-simultaneous-winding}, and from the fact that both parts commute, since the fibration is totally geodesic.
\end{proof}

\begin{theorem}\label{thm:simulteneous-winding-Stiefel-manifold}
     Let $\mu\in\mathbb{R}^n_{>0}$ and let $(Z(t))_{t\geq 0} =(Z_1 (t),\dots ,Z_{k+1}(t))_{t\geq 0}$ be a Brownian motion on $V_{n,m}(\mathbb{C})$ such that $Z_j(0)\in\mathbb{C}^{m\times m}\backslash\{ 0\}$ for $1\leq j\leq k+1$.
     Let
     \begin{equation*}
         (\eta_1 (t),\dots ,\eta_{k+1} (t))_{t\geq 0}
     \end{equation*}
     be a continuous stochastic process such that
     \begin{equation*}
         \det (Z_j(t)) =|\det (Z_j(t))|e^{i\eta_j (t)/m}\textrm{ for } 1\leq j\leq k+1.
     \end{equation*}
    The following convergence holds in distribution
    \begin{equation*}
        \frac{1}{t}(\eta_1 (t),\dots ,\eta_{k+1} (t))\rightarrow (C_{m(n-m)}^1,\dots , C_{m(n-m)}^{k+1})
    \end{equation*}
    as $t\rightarrow\infty$, where $C_{m(n-m)}^1,\dots ,C_{m(n-m)}^{k+1}$ are independent Cauchy random variables with parameter $m(n-m)$.
\end{theorem}
\begin{proof}
    By theorem \ref{thm:Brownian-motion-Stiefel-manifold}, one has in distribution
    \begin{equation*}
        (\eta_1 (t),\dots ,\eta_{k+1}(t)) =(m\beta_1(t) +\mathfrak{a}_1(t),\dots ,m\beta_{k+1}(t) +\mathfrak{a}_{k+1}(t)) ,
    \end{equation*}
    where $(\mathfrak{a} (t))_{t\geq 0}$ is the stochastic area process of a Brownian motion on $F_{m,2m,\dots ,km}(\mathbb{C}^n)$ and $(\beta (t))_{t\geq 0}$ is an independent Brownian motion on $\mathbb{R}^{k+1}$.
    Since $\frac{1}{t}\beta (t)$ almost surely converges to $0$ as $t\rightarrow\infty$, the result follows from corollary \ref{cor:limit-area-process-on-Fm}.
\end{proof}

\begin{remark}
    Just like in \cite[\textsection 5]{baudoin2025fullflag}, the above limit theorem also holds for the canonical variation of the fibration \eqref{eq:fibration-simultaneous-winding}, which we will now briefly describe.
    Let $\frac{\partial}{\partial\theta_1} ,\dots ,\frac{\partial}{\partial\theta_{k+1}}$ be the generators of the corresponding group action.
    We can decompose $g_{V_{n,m}(\mathbb{C})} =g_{\mathcal{H}}\oplus g_{\mathcal{V}}$, where $\mathcal{V}$ is the subbundle spanned by the generators $\frac{\partial}{\partial\theta_j}$ for $1\leq j\leq k+1$ of the group action and $\mathcal{H}$ its orthogonal complement.
    Given $\mu =(\mu_1,\dots ,\mu_{k+1} )$ with $\mu_j >0$ for $1\leq j\leq k+1$, define $g^{\mu}_{\mathcal{V}}\left(\frac{\partial}{\partial\theta_{\ell}} ,\frac{\partial}{\partial\theta_{j}}\right) :=(\mu_{\ell}\mu_{j})^{-1}\delta_{\ell j}$.
    The canonical variation of the fibration \eqref{eq:fibration-simultaneous-winding} is $V_{n,m}(\mathbb{C})$ together with the metric $g_{\mu} :=g_{\mathcal{H}}\oplus g^{\mu}_{\mathcal{V}}$.
\end{remark}

\bibliographystyle{plain}
\bibliography{reference.bib}

\end{document}